\documentclass{amsart}

\usepackage[usenames,dvipsnames]{xcolor}
\usepackage{amsthm}
\usepackage{amssymb}
\usepackage{amsfonts}
\usepackage{fancyhdr}
\usepackage{tikz}
\usepackage{hyperref}
\usepackage[normalem]{ulem}
\usepackage{graphicx}
\usepackage{caption}
\usepackage{cite}
\usepackage{subcaption}
\usepackage{amsmath,amscd}
\usepackage[mathscr]{euscript}
\usepackage{IEEEtrantools}

\newcommand{\Path}{{\mathscr{P}}}
\newcommand{\rf}{{\mathrm{rf_{\diagdown}}}}
\newcommand{\rfu}{{\mathrm{rf_{\diagup}}}}
\newcommand{\rp}{{\mathrm{rp}}}
\newcommand{\rv}{{\mathrm{rv}}}
\newcommand{\st}{{\mathrm{st}}}
\definecolor{light-gray}{gray}{0.8}
\definecolor{mgray}{gray}{.6}

\newtheorem{theorem}{Theorem}[section]
\newtheorem{corollary}[theorem]{Corollary}
\newtheorem{lemma}[theorem]{Lemma}
\newtheorem{proposition}[theorem]{Proposition}

\begin{document}
\date{}
\title{The Hopf Algebra of Generic Rectangulations}
\thanks{The author supported in part by NSF grants DMS-0943855 and CCF-1017217.} 
\author{Emily Meehan}

\maketitle

\begin{abstract}
A family of permutations called 2-clumped permutations forms a basis for 
a sub-Hopf algebra of the Malvenuto-Reutenauer Hopf algebra of permutations.
The 2-clumped permutations are in bijection with certain decompositions of a square into rectangles, called generic rectangulations.
Thus, we can describe the Hopf algebra of 2-clumped permutations using generic rectangulations (we call this isomorphic Hopf algebra the Hopf algebra of generic rectangulations).
In this paper, we describe the cover relations in a lattice of generic rectangulations that is a lattice quotient of the right weak order on permutations.
We then use this lattice to describe the product and coproduct operations in the Hopf algebra of generic rectangulations.
\end{abstract}

\tableofcontents

\section{Introduction}

A \emph{rectangulation} $R$ of size $n$ is a tiling of a square $S$ by $n$ rectangles.  
If no four tiles of a rectangulation share a vertex, then we say that the rectangulation is a \emph{generic rectangulation}.  
In this paper, we consider generic rectangulations up to combinatorial equivalence.  
Specifically, we say that generic rectangulation $R_1$ is combinatorially equivalent to generic rectangulation $R_2$ if there exists a homeomorphism of the square $S$ that fixes the vertices of $S$ and sends $R_1$ to $R_2$.

In \cite
{grec}, Reading describes a map $\gamma$ from permutations of $[n]=\{1,...,n\}$ to generic rectangulations of size $n$.
The map $\gamma$, described in Section \ref{sect:gamma}, restricts to a bijection 
between certain pattern-avoiding permutations called 2-clumped permutations, defined in Section \ref{sect:clumped}, and generic rectangulations.  
The set of all 2-clumped permutations forms a basis for a Hopf subalgebra of the Malvenuto-Reutenauer Hopf algebra of permutations.
(For a description of the Malvenuto-Reutenauer Hopf algebra of permutations, see  \cite{Malv, MalvReut}.)
We call the Hopf algebra of 2-clumped permutations~$Cl^2$ and use $Cl_n^2$ to denote the set of 2-clumped permutations of size $n$. 
The product and coproduct operations in $Cl^2$, which we denote  respectively by $\bullet_{Cl^2}$ and $\Delta_{Cl^2}$, can be defined using the corresponding operations in the Malvenuto-Reutenauer Hopf algebra of permutations and then eliminating elements not in the Hopf algebra of 2-clumped permutations. 
However, to better describe the subalgebra, in this paper we characterize it as a Hopf algebra of generic rectangulations. 

Let $gRec$ denote the Hopf algebra of generic rectangulations that is isomorphic to $Cl^2$ via~$\gamma$ and let $\bullet_{gR}$ and $\Delta_{gR}$ respectively denote the product and coproduct operations in~$gRec$.
We denote the set of all generic rectangulations of size $n$ by~$gRec_n$.
Given two generic rectangulations $R_1\in gRec_p$ and $R_2\in gRec_q$ we will describe $R_1 \bullet_{gR} R_2$ as the sum of the elements in an interval of a lattice on~$gRec_{p+q}$.
This is analogous to the description of the product in the Malvenuto-Reutenauer Hopf algebra of permutations as a sum of all elements in an interval of the right weak order on permutations of $[p+q]$ (defined below).
Our first main result will describe this lattice on $gRec_{n}$ in terms of the combinatorics of generic rectangulations.
Before providing this description, we explain the relationship between this lattice and the right weak order on permutations.
Let $S_n$ denote the permutations of $[n]$.
Let $x,y \in S_n$, where~$x$ and~$y$ are represented in one-line notation.
Then $x \lessdot y$ in the \emph{right weak order} on~$S_n$ if and only if~$y$ can be obtained from $x=x_1 \cdots x_n$ by interchanging entries~$x_i$ and~$x_{i+1}$ of $x$ where $x_i<x_{i+1}$.  
The fibres of the map $\gamma$ from $S_n$ to $gRec_n$ define a lattice congruence on the right weak order.
The natural isomorphism from the quotient of the right weak order on $S_n$ (modulo this congruence) to the set of generic rectangulations defines a lattice structure on $gRec_n$. 
Reusing notation, we also let $gRec_n$ denote this partial order on generic rectangulations of size $n$. 
In our description of the lattice $gRec_n$, we use two types of local moves, called generic pivots and  wall slides, illustrated by the five diagrams in Figure \ref{fig:moves}.  

To describe these local moves, we require a few additional definitions.
We call a point in $S$ a \emph{vertex} of $R$ if the point is the vertex of some rectangle of $R$.  
An \emph{edge} of $R$ is a line segment contained in the side of some rectangle of $R$ such that the endpoints of the line segment are vertices and the segment has no vertices in its interior.
 A maximal union of edges forming a line segment is a \emph{wall} of $R$.   
 
The right two diagrams of Figure \ref{fig:moves} show wall slides.
Given a vertical wall $W$ of $R$, a \emph{vertical wall slide} switches the order of two walls incident to the interior of $W$.  
Let $W_l$ and~$W_r$ be walls of $R$ incident to the interior of $W$ such that~$W_l$ extends to the left of~$W$, wall~$W_r$ extends to the right of~$W$ and no other wall incident to~$W$ has endpoint between the endpoints of $W_l$ and $W_r$. 
A wall slide performed on $W_l$ and $W_r$ switches their relative orders along $W$ and results in a new generic rectangulation.
Similarly, a \emph{horizontal wall slide} switches the order of two walls incident to a horizontal wall $W$ and results in a new generic rectangulation.  
If~$W_u$ is incident to $W$, extending up from $W$, and $W_d$ is incident to~$W$, extending down from~$W$, such that no other walls incident to $W$ have endpoints between the endpoints of~$W_u$ and~$W_d$, then switching the order of $W_u$ and $W_d$ on $W$ is a horizontal wall slide. 
  
The precise definition of a generic pivot is more complicated than that of a wall slide. 
We call an edge that can participate in a generic pivot a \emph{pivotable edge}.
A \emph{generic pivot} replaces a pivotable vertical (or horizontal) edge of a  generic rectangulation with a distinct horizontal (or vertical) edge resulting in a new generic rectangulation.  
The left three diagrams of Figure~\ref{fig:moves} illustrate the three types of generic pivots.
In each case, a segment separating two rectangles is removed and replaced with a segment that produces a distinct generic rectangulation.
The dashed segments of each diagram indicate edges to which no additional segments of $R$ may be incident.
If a segment of $R$ is incident to a dashed edge, then the edge separating the two rectangles is not pivotable.  
In this case, a wall slide or sequence of wall slides must move the edges incident to the dashed segments before the generic pivot can occur.   
When a generic pivot is performed, the new edge introduces new vertice(s) along some wall(s) of $R$ and these vertice(s) must be placed with respect to the other vertices already on that wall so that no edges are incident to dashed segments in the new rectangulation.

We now state our first main result.

\begin{theorem}
\label{thm: covers}
Let $R_1$ and $R_2$ be generic rectangulations of size $n$.  Then $R_1 \lessdot R_2$ in $gRec_n$ if and only if :
\begin{itemize}
\item $R_1$ and $R_2$ are related by a generic pivot such that the pivoted edge is vertical in $R_1$~or
\item $R_1$ and $R_2$ are related by a single wall slide as shown in the two rightmost diagrams of Figure \ref{fig:moves}.
\end{itemize}
\end{theorem}

\begin{figure}
\begin{tikzpicture}[scale=.68]
\draw (0,0) -- (0,2);
\draw[dashed, ultra thick] (0,2) -- (1,2);
\draw (1,2) -- (2,2) -- (2,0);
\draw[dashed, ultra thick] (2,0) -- (1,0);
\draw (1,0) -- (0,0);
\draw (1,0) -- (1,2);

\draw (1,2.5) -- (0.75,3);
\draw (1,2.5) -- (1.25, 3);
\draw[fill] (1, 3) circle [radius=1pt];

\draw (0,3.5) -- (0,4.5);
\draw[dashed, ultra thick] (0,4.5) -- (0,5.5);
\draw (0,5.5) -- (2,5.5) -- (2,4.5);
\draw[dashed, ultra thick] (2,4.5) -- (2,3.5);
\draw (2,3.5) -- (0,3.5);
\draw (0,4.5) -- (2,4.5);

\draw (4,0) rectangle (5,2);
\draw (5,1) -- (6,1)--(6,0);
\draw[dashed, ultra thick] (6,0) -- (5,0);

\draw (5,2.5) -- (4.75,3);
\draw (5,2.5) -- (5.25, 3);
\draw[fill] (5, 3) circle [radius=1pt];

\draw (4,3.5) rectangle (6,4.5);
\draw[dashed, ultra thick] (4, 4.5)--(4, 5.5);
\draw(4, 5.5) -- (5, 5.5)-- (5, 4.5);

\draw (9,0) rectangle (10,2);
\draw (9, 1) -- (8, 1) -- (8, 2);
\draw[dashed, ultra thick](8,2) --(9,2);

\draw (9,2.5) -- (8.75,3);
\draw (9,2.5) -- (9.25, 3);
\draw[fill] (9, 3) circle [radius=1pt];

\draw (8, 4.5) rectangle (10, 5.5);
\draw (9, 4.5) -- (9, 3.5) -- (10, 3.5); 
\draw[dashed, ultra thick] (10, 3.5) -- (10, 4.5);

\draw (13, 0) -- (13, .5);
\draw[dashed, ultra thick] (13, .5) -- (13, 1.5);
\draw (13, 1.5) -- (13, 2);
\draw (12, .5) -- (13, .5);
\draw (13, 1.5) -- (14, 1.5);

\draw (13,2.5) -- (12.75,3);
\draw (13,2.5) -- (13.25, 3);
\draw[fill] (13, 3) circle [radius=1pt];

\draw (13, 3.5) -- (13, 4);
\draw[dashed, ultra thick] (13, 4) -- (13, 5);
\draw (13, 5) -- (13, 5.5);
\draw (13, 4) -- (14, 4);
\draw (12, 5) -- (13, 5);

\draw(16, 1) -- (16.5, 1);
\draw[dashed, ultra thick] (16.5, 1) -- (17.5, 1);
\draw (17.5, 1) -- (18, 1);
\draw (16.5, 1) -- (16.5, 2);
\draw (17.5, 1) -- (17.5, 0);

\draw (17,2.5) -- (16.75,3);
\draw (17,2.5) -- (17.25, 3);
\draw[fill] (17, 3) circle [radius=1pt];

\draw(16, 4.5) -- (16.5, 4.5);
\draw[dashed, ultra thick] (16.5, 4.5) -- (17.5, 4.5);
\draw (17.5, 4.5) -- (18, 4.5);
\draw (16.5, 4.5) -- (16.5, 3.5);
\draw (17.5, 4.5) -- (17.5, 5.5);
\end{tikzpicture}
\caption{Every cover relation in $gRec_n$ is obtained by performing one of the local changes shown in the figure on a generic rectangulation.  
If the configuration is to participate in the illustrated move, no edge of the generic rectangulation can have an endpoint in the interior of a dashed segment.}
\label{fig:moves}
\end{figure}
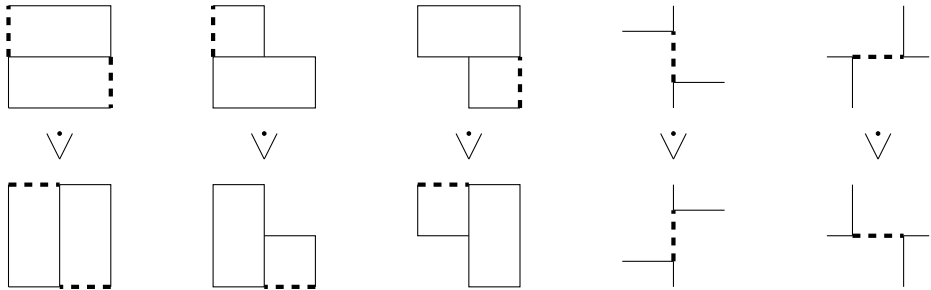

Figure \ref{fig:cover seq} shows several examples of the cover relations described in Theorem \ref{thm: covers}.
Note that each rectangle in the figure is labeled by an element of $\{1,...,7\}$.
The map $\gamma$, which provides a numbering of the rectangles in a generic rectangulation, is described in Section \ref{sect:gamma}.
Let \emph{rectangle~$i$} refer to the rectangle labeled with~$i$.
In the first rectangulation of Figure \ref{fig:cover seq}, a generic pivot cannot be performed on the edge separating the shaded rectangles since the edge separating rectangles 1 and 2 is incident to the interior of the upper segment (or top) of rectangle 3. 
Performing a horizontal wall slide on the bold edges of the first rectangulation of the sequence, we obtain the second rectangulation.
A generic pivot can then be performed on the edge separating rectangles 3 and 4 in the second rectangulation of the sequence to obtain the third rectangulation. 
To obtain the fourth rectangulation of the sequence, a generic pivot is performed on the edge between rectangles 2 and 5.  
Performing the pivot introduces a new vertex along the wall separating rectangles~5 and 7.  
To avoid having an edge incident to the right side of rectangle 5 in the fourth rectangulation (as is disallowed in Figure~\ref{fig:moves}), the left vertex of the edge separating rectangles 6 and 7 is placed above the right vertex of the edge separating rectangles~2 and 5.
This is possible because, before performing the generic pivot on the edge separating rectangles 2 and 5, the edge between rectangles 6 and 7 can be moved up without changing the equivalence class of the generic rectangulation.

\begin{figure}
\begin{tikzpicture}[scale=.33]

\draw(0,0)--(0,7)-- (7,7)-- (7,0)--(0,0);
\draw[fill=light-gray] (0,0) rectangle (3,5);
\draw[fill=light-gray] (3,0) rectangle (4,5);
\draw[line width = 2pt](0,5)--(4,5);
\draw(7,1)--(5,1);
\draw[line width = 2pt](1,7)--(1,5);
\draw[line width = 2pt](3,5)--(3,0);
\draw(4,7)--(4,0);
\draw(5,7)--(5,0);
\node(1) at (.5,6.5) {$1$};
\node(2) at (1.5,5.6) {$2$};
\node(3) at (2.5,4.4) {$3$};
\node(4) at (3.5,3.5) {$4$};
\node(5) at (4.5,2.5) {$5$};
\node(6) at (5.5,1.5) {$6$};
\node(7) at (6.5,.5) {$7$};

\node(perm) at (3.5, -1) {$\gamma(3{\bf14}2576)$};

\draw(8,3.5)--(9,4);
\draw(8,3.5)--(9, 3);
\draw[fill] (9, 3.5) circle [radius=.9pt];

\draw(10,0)--(10,7)-- (17,7)-- (17,0)--(10,0);
\draw[line width = 2pt](10,0) rectangle (11,5);
\draw[line width = 2pt](11,0) rectangle (14,5);
\draw(10,0)--(14,0);
\draw(10,5)--(14,5);
\draw(17,1)--(15,1);
\draw(10,0)--(10,5);
\draw(13,7)--(13,5);
\draw(11,5)--(11,0);
\draw(14,7)--(14,0);
\draw(15,7)--(15,0);
\draw(14,0)--(14,5);
\node(1) at (12.5,6.5) {$1$};
\node(2) at (13.5,5.6) {$2$};
\node(3) at (10.5,4.4) {$3$};
\node(4) at (11.5,3.5) {$4$};
\node(5) at (14.5,2.5) {$5$};
\node(6) at (15.5,1.5) {$6$};
\node(7) at (16.5,.5) {$7$};

\node(perm) at (13.5, -1) {$\gamma({\bf34}12576)$};

\draw(18,3.5)--(19,4);
\draw(18,3.5)--(19, 3);
\draw[fill] (19, 3.5) circle [radius=.9pt];

\draw(20,0)--(20,7)-- (27,7)-- (27,0)--(20,0);
\draw[line width = 2pt](23, 5) rectangle (24,7);
\draw[line width = 2pt](24,0) rectangle (25, 7);
\draw(20,5)--(24,5);
\draw(27,1)--(25,1);
\draw(20,4)--(24,4);
\draw(23,7)--(23,5);
\draw(24,7)--(24,0);
\draw(25,7)--(25,0);
\node(1) at (22.5,6.5) {$1$};
\node(2) at (23.5,5.6) {$2$};
\node(3) at (20.5,4.5) {$3$};
\node(4) at (21.5,3.5) {$4$};
\node(5) at (24.5,2.5) {$5$};
\node(6) at (25.5,1.5) {$6$};
\node(7) at (26.5,.5) {$7$};

\node(perm) at (23.5, -1) {$\gamma(431{\bf25}76)$};

\draw(28,3.5)--(29,4);
\draw(28,3.5)--(29, 3);
\draw[fill] (29, 3.5) circle [radius=.9pt];

\draw(30,0)--(30,7)-- (37,7)-- (37,0)--(30,0);
\draw(30,3)--(35,3);
\draw(37,6)--(35,6);
\draw(30,2)--(34,2);
\draw(33,7)--(33,3);
\draw(34,3)--(34,0);
\draw(35,7)--(35,0);
\node(1) at (32.5,4.5) {$1$};
\node(2) at (33.5,3.5) {$2$};
\node(3) at (30.5,2.5) {$3$};
\node(4) at (31.5,1.5) {$4$};
\node(5) at (34.5,.5) {$5$};
\node(6) at (35.5,6.5) {$6$};
\node(7) at (36.5,5.5) {$7$};

\node(perm) at (33.5, -1) {$\gamma(4315276)$};
\end{tikzpicture}
\caption{A sequence of cover relations in $gRec_7$.}
\label{fig:cover seq}
\end{figure}
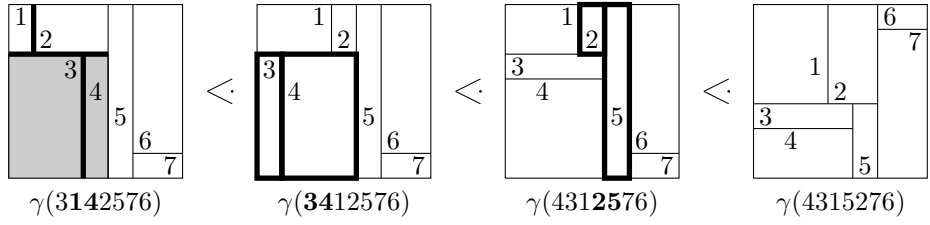



Having described the lattice $gRec_n$, we use this lattice to describe $\bullet_{gR}$, the product operation in the Hopf algebra $gRec$.
Given generic rectangulations $R_1$ and~$R_2$, let $R_1R_2'$ denote the \emph{horizontal concatenation} of $R_1$ and $R_2$.
This is a generic rectangulation obtained by first placing $R_1$ adjacent to~$R_2$ so that the right side of $R_1$ coincides with the left side of $R_2$.
The resulting figure is rescaled so that the outer boundary of $R_1 \cup R_2$ is a square and wall slides are performed on the shared wall so that all edges extending left from the shared wall are below all edges extending right from the shared wall. 
Let $R_2'R_1$ denote the \emph{vertical concatenation} of $R_1$ and $R_2$ which is obtained by placing $R_1$ adjacent to $R_2$ so that the top of~$R_2$ coincides with the bottom of $R_1$, rescaling, and then performing wall slides along the shared wall so that all edges extending down from the wall are left of all edges extending up from the wall.
 Examples of a horizontal and a vertical concatenation are shown in Figure~\ref{fig:concat}.
 The numbering of the rectangles in the figure again comes from the map~$\gamma$, defined in Section~\ref{sect:gamma}.
 We denote the horizontal and vertical concatenations by $R_1R_2'$ and $R_2'R_1$ respectively because this notation mimics the notation used in Section \ref{sect:clumped} to discuss related permutations. 
 
 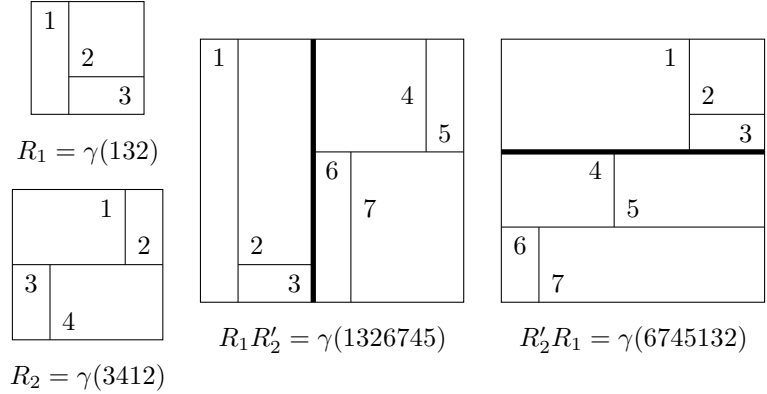
\begin{figure}
 \begin{tikzpicture}[scale=.5]
 \draw(0,0) rectangle (4,4);
 \draw(0,2)--(4,2);
 \draw(1,0)--(1,2);
 \draw(3,2)--(3,4);
 \node at (2,-1) {$R_2=\gamma(3412)$};
 \node at (.5,1.5) {$3$};
  \node at (1.5,.5) {$4$};
 \node at (2.5, 3.5) {$1$};
 \node at (3.5,2.5) {$2$};
 
 \draw(.5,6) rectangle (3.5,9);
 \draw(1.5,6)--(1.5, 9);
 \draw(1.5, 7)--(3.5, 7);
 \node at (2,5) {$R_1=\gamma(132)$}; 
 \node at (1, 8.5) {$1$};
 \node at (2, 7.5) {$2$};
 \node at (3, 6.5) {$3$};
 
 \draw(5,1) rectangle (12, 8);
 \draw(6,1)--(6,8);
 \draw[line width = 2pt](8,1)--(8,8);
 \draw(9,1)--(9,5);
 \draw(6,2)--(8,2);
 \draw(8,5)--(12,5);
 \draw(11,5)--(11,8);
 \node at (8.5, 0) {$R_1R_2'=\gamma(1326745)$};
 \node at (7.5,1.5) {$3$};
 \node at (6.5, 2.5) {$2$};
 \node at (5.5, 7.5) {$1$};
 \node at (8.5, 4.5) {$6$};
 \node at (9.5, 3.5) {$7$};
 \node at (10.5, 6.5) {$4$};
 \node at (11.5, 5.5) {$5$};
 
 \draw (13, 1) rectangle (20, 8);
 \draw(13, 3) -- (20,3);
 \draw(14, 1)--(14, 3);
 \draw[line width = 2pt] (13, 5)--(20, 5);
 \draw (16, 3)--(16, 5);
 \draw (18, 5)--(18, 8);
 \draw (18, 6)--(20,6);
 \node at (16.5, 0) {$R_2'R_1=\gamma(6745132)$};
 \node at (13.5, 2.5) {$6$}; 
 \node at (14.5, 1.5) {$7$};
 \node at (15.5, 4.5) {$4$};
 \node at (16.5, 3.5) {$5$};
 \node at (17.5, 7.5) {$1$};
 \node at (18.5, 6.5) {$2$};
 \node at (19.5, 5.5) {$3$};
  \end{tikzpicture}
  \caption{The horizontal concatenation $R_1R_2'$ and the vertical concatenation~$R_2'R_1$ of generic rectangulations $R_1$ and $R_2$.  In each concatenation, the wall shared by $R_1$ and $R_2'$ is bolded.} 
  \label{fig:concat}
 \end{figure}

    
Our next main result is the following theorem.

\begin{theorem}\label{thm:product}
Let $R_1$ and $R_2$ be generic rectangulations of size $p$ and $q$ respectively such that $p+q=n$.  Then 
\begin{center}
$R_1\bullet_{gR} R_2=\sum [R_1R_2',R_2'R_1]$
\end{center}
where the summation denotes the sum of all elements of $gRec_n$ in the interval $[R_1R_2',R_2'R_1]$.
\end{theorem}

To describe $\Delta_{gR}$, the coproduct in $gRec$, we require several additional definitions.  
Let $R$ be a generic rectangulation and $\Path$ be a path from the top-left corner to the bottom-right corner of $R$, consisting of down and right steps which are edges of $R$.
We say that $\Path$ is a \emph{good path} if it meets the following two conditions: 
\begin{itemize}
\item The interior of no vertical segment of $\Path$ contains vertices $v$ and $v'$ of $R$ such that vertex $v$ is the upper-left vertex of a rectangle of $R$, vertex $v'$ is the lower-right vertex of a rectangle of~$R$ and~$v$ is below~$v'$.  (The left diagram of Figure \ref{fig:path} illustrates this configuration.)
\item The interior of no horizontal segment of $\Path$ contains vertices $h$ and $h'$ of $R$ such that vertex $h$ is the lower-right vertex of a rectangle of $R$, vertex $h'$ is the upper-left vertex of a rectangle of $R$ and~$h$ is left of $h'$.  (The right diagram of Figure \ref{fig:path} illustrates this configuration.)
\end{itemize}

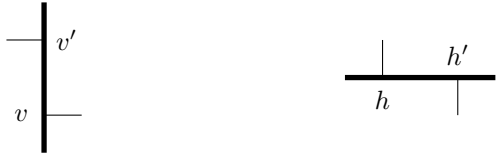
\begin{figure}
\begin{tikzpicture}
\draw[line width = 2pt](0,0)--(0,2);
\draw(0,.5)--(.5,.5);
\draw (0,1.5)--(-.5, 1.5);
\node at (-0.3,.5) {$v$};
\node at (0.3, 1.5) {$v'$};

\draw[line width = 2 pt] (4,1)--(6, 1);
\draw(4.5, 1)--(4.5, 1.5);
\draw (5.5, 1)--(5.5, .5);
\node at (4.5, .7) {$h$};
\node at (5.5, 1.3) {$h'$};
\end{tikzpicture}\caption{Configurations that good paths avoid.} \label{fig:path}
\end{figure}

An example of a good path in a generic rectangulation $R$ is shown as the darkened path in the upper-left diagram of Figure~\ref{fig:completions}.  
In this rectangulation, the path traveling from the upper-left corner of $S$ to the lower-right corner of $S$, passing above rectangles 1, 3, 4, 5, and 8, and below the remaining rectangles is \emph{not} a good path.  
The lower-right vertex of rectangle 4 and the upper-left vertex of rectangle~7, both lying on the interior of a single vertical segment of the path, violate the first condition in the definition of a good path.   

\begin{figure}
\begin{tikzpicture}[scale=.29]
\draw (0,0) rectangle (10,10);
\draw (0,6)--(5,6)--(5,0);
\node at (4.5,5.5) {\small{5}};
\draw (0,8)--(1,8)--(1,6);
\node at (.5,7.5) {\small{3}};
\draw (5,1)--(8,1)--(8,0);
\node at (7.5, .5) {\small{8}};
\draw (5,4)--(8,4)--(8,1);
\node at (6.5, 1.5) {\small{7}};
\draw (1,8)--(5,8)--(5,6);
\node at (1.5,6.6) {\small{4}};
\draw (3,10)--(3,8);
\node at (2.5, 9.4) {\small{1}};
\draw (5,10)--(5,8);
\node at (3.5, 8.6) {\small{2}};
\draw (8,2.9)--(10,2.9);
\node at (9.3, 2.4) {\small{10}};
\draw (8,4)--(10,4);
\node at (8.5, 3.4) {\small{9}};
\node at (5.5, 4.6) {\small{6}};
\node at (5,-1) {$R=\gamma(5387412\hspace{.15em} \scalebox{0.8}{\raisebox{.08em}{10}} \hspace{.2em}  96)$};
\node at (5.6, -2.2) {$=\gamma(538741\hspace{.15em} \scalebox{0.8}{\raisebox{.08em}{10}} \hspace{.2em}  926)$};
\node at (5.6, -3.4) {$=\gamma(587 \hspace{.15em} \scalebox{0.8}{\raisebox{.08em}{10}} \hspace{.2em} 934126)$};
\draw  (8.15,-.935) circle (.62cm);
\draw  (7.5,-2.17) circle (.62cm);
\draw  (5.7,-3.36) circle (.62cm);
\draw[line width = 2pt](0,10)--(3,10)--(3,8)--(5,8)--(5,4)--(10,4)--(10,0);

\draw (15,0) rectangle (25,10);
\draw (15,6)--(20,6); \draw(20,4)--(20,0);
\node at (19.5,5.5) {\small{5}};
\draw (15,8)--(16,8)--(16,6);
\node at (15.5,7.5) {\small{3}};
\draw (20,1)--(23,1)--(23,0);
\node at (22.5, .5) {\small{8}};
\draw (23,4)--(23,1);
\node at (21.5, 1.5) {\small{7}};
\draw (16,8)--(18,8);
\node at (16.5,6.6) {\small{4}};
\node at (17.5, 9.5) {\small{1}};
\draw (23,3)--(25,3);
\node at (24.3, 2.5) {\small{10}};
\node at (23.5, 3.5) {\small{9}};
\node at (20,-1) {$R_l(\Path)$};

\draw (30,0) rectangle (40,10);
\draw (35,10)--(35,8);
\node at (33.5, 8.5) {\small{2}};
\node at (35.5, 8.5) {\small{6}};
\node at (35,-1) {$R_u(\Path)$};

\draw (1,-13) rectangle (9,-5);
\draw (1,-10)--(4,-10)--(4,-13);
\node at (3.5, -10.5) {\small{5}};
\draw (1,-8)--(2,-8)--(2,-10);
\node at (1.5,-8.5) {\small{3}};
\draw(4,-12)--(7,-12)--(7,-13);
\node at (5.5, -12.5) {\small{8}};
\draw (4,-10)--(7,-10)--(7,-12);
\node at (4.5, -11.5) {\small{7}};
\draw (2,-8)--(7,-8)--(7,-10);
\node at (2.5, -9.5) {\small{4}};
\draw (7, -5)--(7,-8);
\node at (6.5, -6.5) {\small{1}};
\draw (7,-6)--(9,-6);
\node at (7.5, -5.5) {\small{9}};
\node at (8.3, -6.5) {\small{10}};
\filldraw  (7, -10) circle (.2cm);
\filldraw (7, -8) circle (.2cm);
\node at (5, -14) {$R_l(\Path)_|=\gamma(\st(538741 \hspace{.15em} \scalebox{0.8}{\raisebox{.08em}{10}} \hspace{.2em} 9))$};
\draw (9.85, -13.85) circle (.62cm);

\draw (16, -13) rectangle (24, -5);
\draw (16, -8)--(18, -8)--(18,-13);
\node at (17.5, -8.5) {\small{5}};
\draw(18, -12)--(20, -12)--(20, -13);
\node at (19.5, -12.5) {\small{8}};
\draw (18, -8)--(20, -8)--(20, -12);
\node at (18.5, -11.5) {\small{7}};
\draw (20, -10)--(24, -10);
\node at (21.5, -10.5) {\small{10}};
\draw (20, -8)--(24,-8);
\node at (20.5, -9.5) {\small{9}};
\draw(16, -6)--(23, -6)--(23, -8);
\node at (22.5, -6.5) {\small{3}};
\draw (23, -6)--(24, -6);
\node at (23.5, -7.5) {\small{4}};
\node at (16.5, -5.5) {\small{1}};
\filldraw  (18, -8) circle (.2cm);
\filldraw(20, -8) circle (.2cm);
\node at (20, -14) {$R_l(\Path)_-=\gamma(\st(587 \hspace{.15em} \scalebox{0.8}{\raisebox{.08em}{10}} \hspace{.2em}  9341))$};
\draw (23.25, -13.92) circle (.62cm);

\draw (34,-10) rectangle (36, -8);
\draw (35, -8)--(35, -10);
\node at (34.5,-8.5) {\small{2}};
\node at (35.5, -9.5){\small{6}};
\node at (35, -11) {$R_u(\Path)_|=R_u(\Path)_-$};
\node at (37.2, -12.3) {$=\gamma(\st(26))$};

\end{tikzpicture}
\caption{Given a good path $\Path$ is a generic rectangulation, we construct the vertical and horizontal completions of $R_l(\Path)$ and $R_u(\Path)$.
}
\label{fig:completions}
\end{figure}
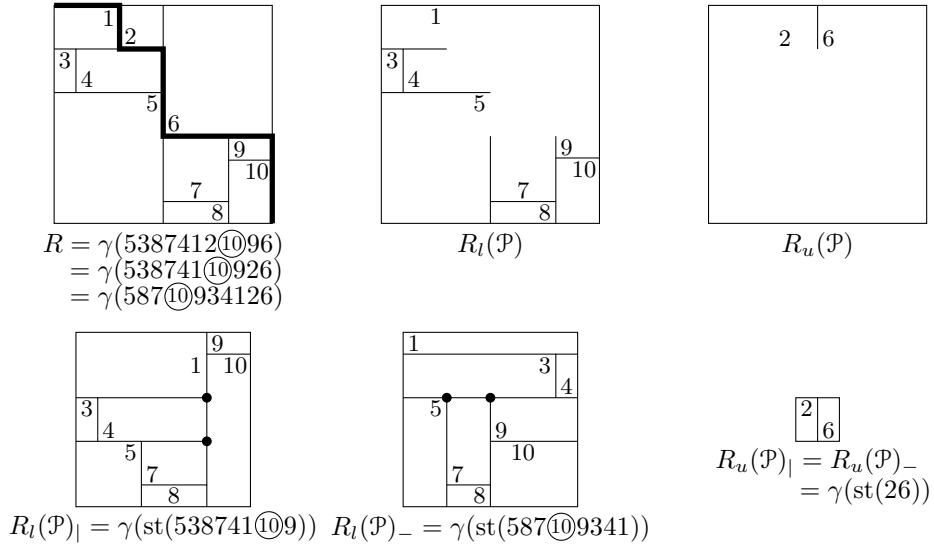

Let $p$ denote the number of rectangles below a good path $\Path$ and $q$ the number of rectangles above $\Path$.
Let $R_l(\Path)$ consist of the edges of $S$ together with the edges of $R$ strictly below~$\Path$ and~$R_u(\Path)$ consist of the edges of $S$ together with the edges of~$R$ strictly above $\Path$, as shown in the example in Figure \ref{fig:completions}.  
We will construct, from $R_l(\Path)$, two generic rectangulations, $R_l(\Path)_|$ and $R_l(\Path)_-$, elements of~$gRec_p$, respectively called the \emph{vertical} and \emph{horizontal completions} of~$R_l(\Path)$.
Similarly, from~$R_u(\Path)$, we will construct  the vertical completion $R_u(\Path)_|$ and horizontal completion $R_u(\Path)_-$, both elements of $gRec_q$.  

The vertical completion~${R_l(\Path)}_|$ is constructed using the following four steps:
\begin{itemize}
\item Each open horizontal edge of $R_l(\Path)$ (i.e. each horizontal edge of $R_l(\Path)$ whose right endpoint lies on $\Path$ in $R$) is extended to the right by $\epsilon$.
\item Each open vertical segment of $R_l(\Path)$ is extended upwards until it meets one of the horizontal edges extended in the previous step or the upper edge of~$S$.  
\item Every horizontal edge extended in the first step is further extended to the right until the extension meets the interior of some vertical edge or the right side of $S$.  Call each new vertex constructed in this step a \emph{constructed vertex}.
\item Along each vertical wall $W$, 
 wall slides changing the order of a vertex of~$R_l(\Path)$ and a constructed vertex are performed until the resulting order meets one of the following conditions: the set of constructed vertices is immediately above the uppermost vertex that is the right endpoint of an edge in $R_l(\Path)$, or if no vertex of $R_l(\Path)$ meets this condition, then wall slides are performed until the constructed vertices are below all other vertices on~$W$. 
\end{itemize}
In the example shown in the lower left diagram of Figure \ref{fig:completions}, the extension of all open horizontal edges of $R_l(\Path)$ in the first step of the construction of $R_l(\Path)_|$ prevents the extension of the left edge of rectangle 7 above the bottom edge of rectangle 4.
In the final step of the construction, wall slides are performed to place the edge separating rectangles 9 and 10 above the constructed vertices (which are enlarged for emphasis). 
After these wall slides, the constructed vertices are immediately above the right endpoint of the edge between rectangles 7 and 8. 
 
The vertical completion $R_u(\Path)_|$ is similarly constructed, extending horizontal edges to the left rather than to the right, vertical edges down rather than up, and performing slides along each vertical wall $W$ containing constructed vertices so that constructed vertices are immediately below the lowermost vertex that is the left endpoint of an edge in $R_u(\Path)$ or, if no such vertex exists, so that the constructed vertices are above all other vertices on $W$.

The constructions of the horizontal completions are similar.
To construct ${R_l(\Path)}_-$:
\begin{itemize} 
\item Extend upwards by~$\epsilon$ every open vertical edge of~$R_l(\Path)$.
\item Extend to the right each open horizontal edge of~$R_l(\Path)$ until the edge meets a vertical edge.
\item Further extend each vertical edge extended in the first step until the extension meets the interior of some horizontal edge or the top of $S$.  
Call the new vertices constructed in this step constructed vertices.
\item Perform wall slides along each horizontal wall $W$ containing the constructed vertices, changing the order of a constructed vertex and a vertex of $R_l(\Path)$ in each wall slide, until all constructed vertices are immediately to the right of the rightmost vertex that is the upper endpoint of an edge in $R_l(\Path)$, or if no vertex of $R_l(\Path)$ meets this condition, until the constructed vertices are to the left of all other vertices on $W$.  
\end{itemize}

An example of $R_l(\Path)_-$ is shown in the middle diagram of the lower row of Figure~\ref{fig:completions}.
Notice that in this diagram, unlike in $R_l(\Path)_|$, the horizontal wall between rectangles 3 and 5 and the horizontal wall between rectangles 3 and 1 are extended until they reach the right side of~$S$.
Since no edges of $R_l(\Path)$ extend down from the horizontal wall $W$ between rectangles 3 and 5, in the final step of the construction, wall slides are performed until the constructed vertices (again enlarged for emphasis) are to the left of the other vertex on~$W$.
 
We construct ${R_u(\Path)}_-$ by extending vertical segments downward, horizontal edges to the left, and performing wall slides along horizontal walls containing constructed vertices so that all constructed vertices are immediately to the left of the leftmost vertex that is the lower endpoint of an edge in $R_u(\Path)$ or, if no such vertex exists, so that the constructed vertices are right of all other vertices on $W$.  

\begin{theorem}
\label{thm:coproduct}
Let $R\in gRec_n$, 
\begin{center} $I_\Path = \sum [{R_l(\Path)}_|, {R_l(\Path)}_-] \text{ and }
J_\Path = \sum[{R_u(\Path)}_|, {R_u(\Path)}_-]. $ \end{center}
where the summations respectively denote the sum of all elements of $gRec_p$ in the interval $[{R_l(\Path)}_|, {R_l(\Path)}_-]$ and the sum of all elements of $gRec_q$ in the interval $[{R_u(\Path)}_|, {R_u(\Path)}_-]$.
Then 
$$\Delta_{gR}(R)=\sum_{\Path \text{ is good}} I_{\Path}\otimes J_{\Path}.$$
\end{theorem}


\section{The Hopf algebra of 2-clumped permutations}
\label{sect:clumped}
In \cite{grec}, Reading proves that generic rectangulations are in bijection with 2-clumped permutations.
To define $k$-clumped permutations, and in particular the 2-clumped permutations needed in this paper, we first define pattern avoidance. 
Let $p=p_1 \cdots p_l \in S_l$ and $\tilde{p}$ be obtained by inserting a single dash between some adjacent entries of $p$.
We say that a permutation $y\in S_n$ \emph{contains the pattern} $\tilde{p}$ if there exists some subsequence $y_{i_1} \cdots y_{i_l}$ of $y$ with the following two properties.
First, the relative order of the terms in the subsequence matches the relative order of the entries of $p$, i.e. for all $j,k\in [l]$, we have that $y_{i_j}<y_{i_k}$ if and only if $p_j<p_k$.
Secondly, if~$p_j$ and $p_{j+1}$ are not separated by a dash in $\tilde{p}$, then $i_j=i_{j+1}-1$, i.e.~$y_{i_j}$ and~$y_{i_{j+1}}$ are adjacent in $y$.
If $y$ does not contain the pattern $p$, we say that $y$ \emph{avoids}~$p$.
For example, consider $y=546312 \in S_6$.
The subsequence $5612$ is an occurrence of the pattern $3$-$4$-$1$-$2$ in~$y$, but is not an occurrence of the pattern $3$-$41$-$2$ since the~$6$ and $1$ are non-adjacent in $y$.

A pair $y_{i}, y_{i+1}$ of some $y\in S_n$ is a \emph{descent} of $y$ if $y_i>y_{i+1}$. 
For every descent of $y$, we define a \emph{clump} to be a maximal set of consecutive values $a, a+1, ..., b$ with $y_{i+1}<a<b<y_i$ such that in $y$ either all elements of $\{a, a+1, ...,b\}$ occur to the left of the descent or all elements of $\{a, a+1, ..., b\}$ occur to the right of the descent.
The pair $92$ is a descent of the permutation $167439285$.
Four clumps are associated with this descent, $\{3,4\}, \{5\}, \{6,7\},$ and $\{8\}$.
A permutation~$y$ is a \emph{$k$-clumped} permutation if every descent of $y$ has at most $k$ associated clumps.
The permutation~$167439285$ is $k$-clumped for any $k\geq 4$ because four clumps are associated with the descent $92$ and fewer clumps are associated with any other descent of the permutation.

Permutations that avoid the patterns $\{2$-$31, 31$-$2\}$ are $0$-clumped permutations. 
Every descent $y_iy_{i+1}$ in a $0$-clumped permutation satisfies $y_i-y_{i+1}=1$.
There is a bijection between $0$-clumped permutations in $S_n$ and compositions of $n$.
To find the composition of $n$ that corresponds to the $0$-clumped permutation $\sigma=\sigma_1 \cdots \sigma_n$, use~$\sigma$ to record a sequence of pluses and commas.  
Specifically, if $\sigma_i>\sigma_{i+1}$, then the $i^{\text{th}}$ entry of the sequence is a plus.  
Otherwise, the $i^{\text{th}}$ entry of the sequence is a comma.
For example, the permutation $217654398$ corresponds to the sequence $+,++++,+$.
Inserting a 1 between each pair of consecutive entries of this sequence, we obtain $1+1,1+1+1+1+1,1+1$ or the composition $2,5,2$.
In~\cite{drec}, twisted Baxter permutations, permutations that avoid the patterns \mbox{$\{2$-$41$-$3,$} \mbox{$3$-$41$-$2\}$,} are shown to be in bijection with diagonal rectangulations, defined in Section \ref{sect:gamma}.
The twisted Baxter permutations are exactly the 1-clumped permutations. 
The permutations considered in this paper avoid the set of patterns $\theta= \mbox{\{2\text{-}4\text{-}51\text{-}3}, \mbox{4\text{-}2\text{-}51\text{-}3}, 3$-$51$-$2$-$4$, \mbox{$3$-$51$-$4$-$2\}$} and are called 2-clumped permutations.  
For $m,n \in \mathbb{Z}_{\geq 0}$, let $Cl_n^m$ denote the subset of $S_n$ containing all $m$-clumped permutations.
Define $V$ to be the set of all even natural numbers strictly between 1 and $m+3$ and $V^C$ to be the set of all odd natural numbers strictly between 1 and $m+3$. 
We say that $x\in S_n$ avoids $V$-$(m+3)1$-$V^C$ if and only if $x$ avoids $v_1$-$\cdots$-$v_i$-$(m+3) 1$-$v_1'$-$\cdots$-$v_j'$ where $v_1\cdots v_i$ is any permutation of the elements of $V$ and $v_1' \cdots v_j'$ is any permutation of the elements of~$V^C$.
Then $x\in Cl^{m}_n$ if and only if $x\in S_n$ that avoids the patterns $\{V$-$(m+3)1$-$V^C, V^C$-$(m+3)1$-$V\}$.
The union of the elements of $Cl^m_n$ for all $n\in \mathbb{N}$ forms a basis for a Hopf algebra that we call the Hopf algebra of $m$-clumped permutations \cite[Corollary 1.4, Theorem 9.4]{Reading}.

In $y\in S_n$, let $y_{i_1}\cdots y_{i_{m+3}}$ be an occurrence of the pattern $V$-$(m+3)1$-$V^C$ where $y_{i_j}$ and~$y_{i_{j+1}}$ are respectively the $``m+3"$ and $``1"$ of the pattern.
If the permutation $x$ is obtained by switching~$y_{i_j}$ and $y_{i_{j+1}}$ in $y$, then we say that $x$ is obtained from $y$ by a ($V(m+3)1V^C\to V1(m+3)V^C$) \emph{move}.
We similarly define a ($V1(m+3)V^C \to V(m+3)1V^C$) \emph{move}, a ($V^C(m+3)1V \to V^C1(m+3)V$) \emph{move}, and a ($V^C1(m+3)V \to V^C(m+3)1V$) \emph{move}. 
Define $\pi_{\downarrow}^m: S_n \to Cl_n^m$ by $\pi_{\downarrow}^m(y)=x$ if and only if $x$ is the minimal element with respect to the right weak order on $S_n$ that can be obtained from $y$ using a sequence of ($V(m+3)1V^C\to V1(m+3)V^C$) moves and ($V^C(m+3)1V \to V^C1(m+3)V$) moves. 
Such a unique minimal element exists because the map $\pi_{\downarrow}^m$ defines a lattice congruence on the right weak order in which $\pi_{\downarrow}^m(x)\neq x$ if and only if $x$ contains an occurrence of $V(m+3)1V^C$ or $V^C(m+3)1V$ \cite[Theorem 9.3]{Reading}.
 Every congruence class of a lattice congruence on the right weak order is an interval.

Having described the basis elements in the Hopf algebra of $m$-clumped permutations, we now focus on the Hopf algebra of $2$-clumped permutations and describe the operations $\bullet_{Cl^2}$ and $\Delta_{Cl^2}$.
To describe $\bullet_{Cl^2}$, we provide an additional definition:
Let $x\in Cl_p^2$ and $y=y_1 \cdots y_q\in Cl_q^2$. 
Define $y'$ to be the sequence $y_1'\cdots y_q'$ where $y_i'=y_i+p$ for all $i\in [q]$. 
The concatenation of $x$ and $y'$ is denoted by $xy'$.

Specializing \cite[Equation 6]{Shirley} to the Hopf algebra of 2-clumped permutations, we obtain:
\begin{equation*} x \bullet_{Cl^2} y =\sum [xy',\pi_{\downarrow}^2(y'x)] \end{equation*} where the summation denotes the sum of all elements of the right weak order restricted to~$Cl_{p+q}^2$.  
We observe that $y'x \in Cl_{p+q}^2$ so $\pi_{\downarrow}^2(y'x)=y'x$.
We will use the following corollary to prove Theorem \ref{thm:product} in Section \ref{sect:product}.

\begin{corollary}\label{cor:prod}
Let $x \in Cl^2_p$ and $y\in Cl^2_q$.
Then \begin{equation*}\label{eq:prod}x \bullet_{Cl^2} y =\sum [xy',y'x]. \end{equation*} 
\end{corollary}

We now define terms necessary to describe $\Delta_{Cl^2}$.
Given a sequence $a=a_1 \cdots a_n$, of distinct integers, we define the standardization of $a$, denoted by $\st(a)$, to be the unique permutation $x=x_1\cdots x_n\in S_n$ that respects the ordering of the entries of~$a$.
That is, $x_i<x_j$ if and only if $a_i<a_j$.

Let $x\in Cl_n^2$. 
 We say that a subset $T \subseteq [n]$ is \emph{good with respect to x} if there exists some permutation $x'=x_1' \cdots x_n'\in S_n$ such that $\pi_{\downarrow}^2(x')=x$ and~$T=\{x_1',...,x_{|T|}'\}$.
 Given a good set $T$ such that $|T|=p$ and $q=n-p$, let $x_{\min}$ be the minimal element of the right weak order on $S_n$ such that $\pi_\downarrow^2(x_{\min})=x$ and the first $p$ entries of $x_{\min}$ are the elements of $T$.
 Notice that $x_{\min}$ depends on both $x$ and the selected set $T$ which is good with respect to $x$.
 Define $x_{\min}|_T$ to be the ordering of the elements of~$T$ as they appear in $x_{\min}$.
 Let $x_{\max}$ be the maximal element of the right weak order on $S_n$ such that $\pi_{\downarrow}^2(x_{\max})=x$ and the first~$p$ entries of $x_{\max}$ are the elements of $T$.
 The ordering of the elements of $T$ as they appear in $x_{max}$ is denoted by $x_{\max}|_T$.
 Letting $T^C=[n]-T$, we similarly define $x_{\min}|_{T^C}$ and $x_{\max}|_{T^C}$.
 The following theorem, which will be used to prove Theorem \ref{thm:coproduct} in Section \ref{sect:product}, is a specialization of \cite[Theorem 1.3]{Shirley}.

\begin{theorem}\label{thm:Shcoprod}
Given $x\in Cl_n^2$,
\begin{equation*}\label{eqn:coprod}
\Delta_{Cl^2}(x)=\sum_{T \text{ is good}} I_T \otimes J_T
\end{equation*}
where $I_T$ is the sum of the elements in the interval $[\st(x_{\min}|_T), \pi_{\downarrow}^2 (\st(x_{\max}|_T))]$ of the right weak order on $S_p$ restricted to $Cl_p^2$ and $J_T$ is the sum of elements in the interval $[\st(x_{\min}|_{T^C}), \pi_{\downarrow}^2 (\st(x_{\max}|_{T^C}))]$ of the right weak order on $S_q$ restricted to~$Cl_q^2$. 
 \end{theorem}

\section{The map from permutations to generic rectangulations}
\label{sect:gamma}
Having defined 2-clumped permutations, we now describe the map $\gamma$ from permutations to generic rectangulations which restricts to a bijection between 2-clumped permutations and generic rectangulations.
The map $\gamma: S_n \to gRec_n$ is described in \cite[Section 3]{grec} in two parts: we first describe a map $\rho$ from $S_n$ to the set $dRec_n$ of diagonal rectangulations of size~$n$, and then we perform wall slides to obtain an element of $gRec_n$.
As with generic rectangulations, we consider diagonal rectangulations up to combinatorial equivalence.
A \emph{diagonal rectangulation}~$D\in dRec_n$ is a tiling of the square $S$ with $n$ rectangles such that, for some representative of the rectangulation, the interior of each rectangle of the tiling intersects the diagonal of~$S$ extending from the upper-left corner to the lower-right corner of $S$.
Diagonal rectangulations of size $n$ form a subset of generic rectangulations of size $n$.
 Specifically, a generic rectangulation $R$ is a diagonal rectangulation if and only if there exists some generic rectangulation $R'$ combinatorially equivalent to $R$ such that each rectangle of~$R'$ intersects the upper/left to lower/right diagonal of $R'$.  

Let $x=x_1 \cdots x_n \in S_n$.  
To construct the diagonal rectangulation $\rho(x)$, write~1 through~$n$ along the diagonal of $S$ starting with 1 in the upper-left corner of $S$ and finishing with~$n$ in the lower-right corner.
Place dots along the diagonal between every pair of adjacent numbers.
Each of these numbers will be contained in the interior of some rectangle of $\rho(x)$ (and of $\gamma(x)$), with exactly one number in each rectangle.
We refer to the rectangle containing $i$ as \emph{rectangle $i$}. 
Reading~$x$ from left to right, we construct~$\rho(x)$ recursively as illustrated by the construction of $\rho(53417286)$ in Figure~\ref{fig:rho}.
Let $T_{i-1}$ denote the union of the rectangles labeled by elements of $\{x_1,...,x_{i-1}\}$ and the bottom and left side of $S$.  
Then~$T_{i}$ is constructed using~$T_{i-1}$ and $x_{i}$.
Let $p$ be the dot directly below and right of label $x_i$.
If $p$ is contained in~$T_{i-1}$, then the lower-right vertex of rectangle $x_i$ is the rightmost point of $T_{i-1}$ directly right of $p$.  
If $p$ is not contained in a segment of $T_{i-1}$, then the lower-right vertex of rectangle~$x_i$ is the uppermost point of $T_{i-1}$ directly below $p$.
Let $p'$ be the dot directly above and left of label $x_i$. 
If $p'$ is contained in $T_{i-1}$, then the upper-left vertex of rectangle $x_i$ is the uppermost point of $T_{i-1}$ directly above $p'$.  
If~$p'$ is not contained in a segment of $T_{i-1}$, then the upper-left vertex of rectangle~$x_i$ is the rightmost point of $T_{i-1}$ directly to the left of $p'$.  

In the example shown in Figure \ref{fig:rho}, since the dot $p$ directly below and right of $x_2=3$ is \emph{not} contained in $T_1$, the lower-right vertex of rectangle 3 is the uppermost point of $T_1$ directly below $p$.
Since the dot $p'$ directly above and left of $x_3=4$ is contained in $T_2$, the upper-left vertex of rectangle 4 is the uppermost point of $T_2$ directly above $p'$.

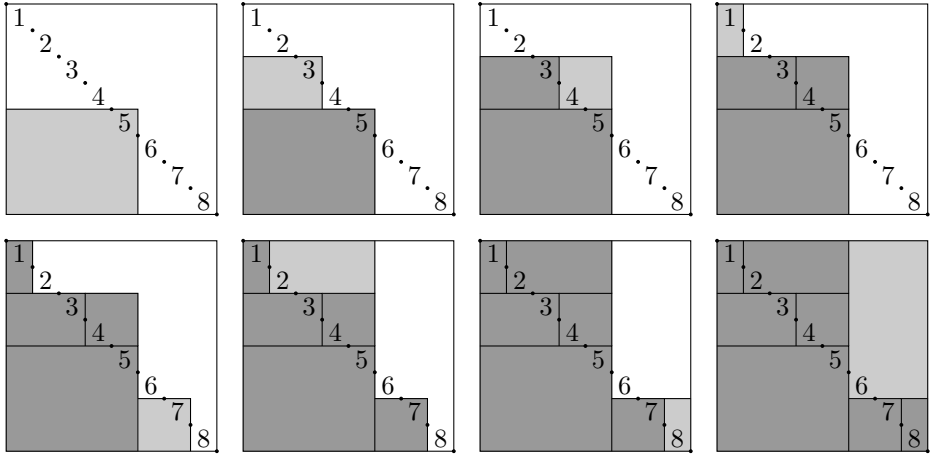
\begin{figure}
\begin{tikzpicture}[scale=.35]
\filldraw[light-gray] (0,0) rectangle (5,4);
\draw(0,0) rectangle (8,8);
\filldraw(0,8) circle [radius=1.5pt] ;
\filldraw(1,7) circle [radius=1.5pt] ;
\filldraw(2,6) circle [radius=1.5pt] ;
\filldraw(3,5) circle [radius=1.5pt] ;
\filldraw(4,4) circle [radius=1.5pt] ;
\filldraw(5,3) circle [radius=1.5pt] ;
\filldraw(6,2) circle [radius=1.5pt] ;
\filldraw(7,1) circle [radius=1.5pt] ;
\filldraw(8,0) circle [radius=1.5pt] ;
\node at (.5,7.5) {$1$};
\node at (1.5,6.5) {$2$};
\node at (2.5,5.5) {$3$};
\node at (3.5,4.5) {$4$};
\node at (4.5,3.5) {$5$};
\node at (5.5,2.5) {$6$};
\node at (6.5,1.5) {$7$};
\node at (7.5,.5) {$8$};
\draw(0,4)--(5,4)--(5,0);

\filldraw[mgray] (9,0) rectangle (14,4);
\filldraw[light-gray] (9, 4) rectangle (12, 6);
\draw(9,0) rectangle (17,8);
\filldraw(9,8) circle [radius=1.5pt] ;
\filldraw(10,7) circle [radius=1.5pt] ;
\filldraw(11,6) circle [radius=1.5pt] ;
\filldraw(12,5) circle [radius=1.5pt] ;
\filldraw(13,4) circle [radius=1.5pt] ;
\filldraw(14,3) circle [radius=1.5pt] ;
\filldraw(15,2) circle [radius=1.5pt] ;
\filldraw(16,1) circle [radius=1.5pt] ;
\filldraw(17,0) circle [radius=1.5pt] ;
\node at (9.5,7.5) {$1$};
\node at (10.5,6.5) {$2$};
\node at (11.5,5.5) {$3$};
\node at (12.5,4.5) {$4$};
\node at (13.5,3.5) {$5$};
\node at (14.5,2.5) {$6$};
\node at (15.5,1.5) {$7$};
\node at (16.5,.5) {$8$};
\draw(9,4)--(14,4)--(14,0);
\draw(9,6)--(12, 6)--(12, 4);

\filldraw[mgray] (18,0) rectangle (23,4);
\filldraw[mgray] (18, 4) rectangle (21, 6);
\filldraw[light-gray] (21,4) rectangle (23,6);
\draw(18,0) rectangle (26,8);
\filldraw(18,8) circle [radius=1.5pt] ;
\filldraw(19,7) circle [radius=1.5pt] ;
\filldraw(20,6) circle [radius=1.5pt] ;
\filldraw(21,5) circle [radius=1.5pt] ;
\filldraw(22,4) circle [radius=1.5pt] ;
\filldraw(23,3) circle [radius=1.5pt] ;
\filldraw(24,2) circle [radius=1.5pt] ;
\filldraw(25,1) circle [radius=1.5pt] ;
\filldraw(26,0) circle [radius=1.5pt] ;
\node at (18.5,7.5) {$1$};
\node at (19.5,6.5) {$2$};
\node at (20.5,5.5) {$3$};
\node at (21.5,4.5) {$4$};
\node at (22.5,3.5) {$5$};
\node at (23.5,2.5) {$6$};
\node at (24.5,1.5) {$7$};
\node at (25.5,.5) {$8$};
\draw(18,4)--(23,4)--(23,0);
\draw(18,6)--(21, 6)--(21, 4);
\draw(21, 6)--(23, 6)--(23, 4);

\filldraw[mgray] (27,0) rectangle (32,4);
\filldraw[mgray] (27, 4) rectangle (30, 6);
\filldraw[mgray] (30,4) rectangle (32, 6);
\filldraw[light-gray] (27, 6) rectangle (28, 8);
\draw(27,0) rectangle (35,8);
\filldraw(27,8) circle [radius=1.5pt] ;
\filldraw(28,7) circle [radius=1.5pt] ;
\filldraw(29,6) circle [radius=1.5pt] ;
\filldraw(30,5) circle [radius=1.5pt] ;
\filldraw(31,4) circle [radius=1.5pt] ;
\filldraw(32,3) circle [radius=1.5pt] ;
\filldraw(33,2) circle [radius=1.5pt] ;
\filldraw(34,1) circle [radius=1.5pt] ;
\filldraw(35,0) circle [radius=1.5pt] ;
\node at (27.5,7.5) {$1$};
\node at (28.5,6.5) {$2$};
\node at (29.5,5.5) {$3$};
\node at (30.5,4.5) {$4$};
\node at (31.5,3.5) {$5$};
\node at (32.5,2.5) {$6$};
\node at (33.5,1.5) {$7$};
\node at (34.5,.5) {$8$};
\draw(27,4)--(32,4)--(32,0);
\draw(27,6)--(30, 6)--(30, 4);
\draw(30, 6)--(32, 6)--(32, 4);
\draw(28,8)--(28,6);
\filldraw[mgray] (0,-9) rectangle (5,-5);
\filldraw[mgray] (0, -5) rectangle (3, -3);
\filldraw[light-gray] (5,-9) rectangle (7, -7);
\filldraw[mgray] (3, -5) rectangle (5,-3);
\filldraw[mgray] (0,-3) rectangle (1,-1);
\draw(0,-9) rectangle (8,-1);
\filldraw(0,-1) circle [radius=1.5pt] ;
\filldraw(1,-2) circle [radius=1.5pt] ;
\filldraw(2,-3) circle [radius=1.5pt] ;
\filldraw(3,-4) circle [radius=1.5pt] ;
\filldraw(4,-5) circle [radius=1.5pt] ;
\filldraw(5,-6) circle [radius=1.5pt] ;
\filldraw(6,-7) circle [radius=1.5pt] ;
\filldraw(7,-8) circle [radius=1.5pt] ;
\filldraw(8,-9) circle [radius=1.5pt] ;
\node at (.5,-1.5) {$1$};
\node at (1.5,-2.5) {$2$};
\node at (2.5,-3.5) {$3$};
\node at (3.5,-4.5) {$4$};
\node at (4.5,-5.5) {$5$};
\node at (5.5,-6.5) {$6$};
\node at (6.5,-7.5) {$7$};
\node at (7.5,-8.5) {$8$};
\draw(0,-5)--(5,-5)--(5,-9);
\draw(0,-3)--(3, -3)--(3, -5);
\draw(5, -7)--(7, -7)--(7, -9);
\draw(3,-3)--(5, -3)--(5, -5);
\draw (1,-3)--(1,-1);
\filldraw[mgray] (9,-9) rectangle (14,-5);
\filldraw[mgray] (9, -5) rectangle (12, -3);
\filldraw[mgray] (14,-9) rectangle (16, -7);
\filldraw[mgray] (12, -5) rectangle (14,-3);
\filldraw[mgray] (9,-3) rectangle (10,-1);
\filldraw[light-gray] (10, -3) rectangle (14, -1);
\draw(9,-9) rectangle (17,-1);
\filldraw(9,-1) circle [radius=1.5pt] ;
\filldraw(10,-2) circle [radius=1.5pt] ;
\filldraw(11,-3) circle [radius=1.5pt] ;
\filldraw(12,-4) circle [radius=1.5pt] ;
\filldraw(13,-5) circle [radius=1.5pt] ;
\filldraw(14,-6) circle [radius=1.5pt] ;
\filldraw(15,-7) circle [radius=1.5pt] ;
\filldraw(16,-8) circle [radius=1.5pt] ;
\filldraw(17,-9) circle [radius=1.5pt] ;
\node at (9.5,-1.5) {$1$};
\node at (10.5,-2.5) {$2$};
\node at (11.5,-3.5) {$3$};
\node at (12.5,-4.5) {$4$};
\node at (13.5,-5.5) {$5$};
\node at (14.5,-6.5) {$6$};
\node at (15.5,-7.5) {$7$};
\node at (16.5,-8.5) {$8$};
\draw(9,-5)--(14,-5)--(14,-9);
\draw(9,-3)--(12, -3)--(12, -5);
\draw(14, -7)--(16, -7)--(16, -9);
\draw(12,-3)--(14, -3)--(14, -5);
\draw (10,-3)--(10,-1);
\draw(14, -3) --(14, -1);
\filldraw[mgray] (18,-9) rectangle (23,-5);
\filldraw[mgray] (18, -5) rectangle (21, -3);
\filldraw[mgray] (23,-9) rectangle (25, -7);
\filldraw[mgray] (21, -5) rectangle (23,-3);
\filldraw[mgray] (18,-3) rectangle (19,-1);
\filldraw[mgray] (19, -3) rectangle (23, -1);
\filldraw[light-gray] (25, -9) rectangle (26, -7);
\draw(18,-9) rectangle (26,-1);
\filldraw(18,-1) circle [radius=1.5pt] ;
\filldraw(19,-2) circle [radius=1.5pt] ;
\filldraw(20,-3) circle [radius=1.5pt] ;
\filldraw(21,-4) circle [radius=1.5pt] ;
\filldraw(22,-5) circle [radius=1.5pt] ;
\filldraw(23,-6) circle [radius=1.5pt] ;
\filldraw(24,-7) circle [radius=1.5pt] ;
\filldraw(25,-8) circle [radius=1.5pt] ;
\filldraw(26,-9) circle [radius=1.5pt] ;
\node at (18.5,-1.5) {$1$};
\node at (19.5,-2.5) {$2$};
\node at (20.5,-3.5) {$3$};
\node at (21.5,-4.5) {$4$};
\node at (22.5,-5.5) {$5$};
\node at (23.5,-6.5) {$6$};
\node at (24.5,-7.5) {$7$};
\node at (25.5,-8.5) {$8$};
\draw(18,-5)--(23,-5)--(23,-9);
\draw(18,-3)--(21, -3)--(21, -5);
\draw(23, -7)--(25, -7)--(25, -9);
\draw(21,-3)--(23, -3)--(23, -5);
\draw (19,-3)--(19,-1);
\draw(23, -3) --(23, -1);
\draw(25, -7)--(26, -7);
\filldraw[mgray] (27,-9) rectangle (32,-5);
\filldraw[mgray] (27, -5) rectangle (30, -3);
\filldraw[mgray] (32,-9) rectangle (34, -7);
\filldraw[mgray] (30, -5) rectangle (32,-3);
\filldraw[mgray] (27,-3) rectangle (28,-1);
\filldraw[mgray] (28, -3) rectangle (32, -1);
\filldraw[mgray] (34, -9) rectangle (35, -7);
\filldraw[light-gray](32, -7) rectangle (35, -1);
\draw(27,-9) rectangle (35,-1);
\filldraw(27,-1) circle [radius=1.5pt] ;
\filldraw(28,-2) circle [radius=1.5pt] ;
\filldraw(29,-3) circle [radius=1.5pt] ;
\filldraw(30,-4) circle [radius=1.5pt] ;
\filldraw(31,-5) circle [radius=1.5pt] ;
\filldraw(32,-6) circle [radius=1.5pt] ;
\filldraw(33,-7) circle [radius=1.5pt] ;
\filldraw(34,-8) circle [radius=1.5pt] ;
\filldraw(35,-9) circle [radius=1.5pt] ;
\node at (27.5,-1.5) {$1$};
\node at (28.5,-2.5) {$2$};
\node at (29.5,-3.5) {$3$};
\node at (30.5,-4.5) {$4$};
\node at (31.5,-5.5) {$5$};
\node at (32.5,-6.5) {$6$};
\node at (33.5,-7.5) {$7$};
\node at (34.5,-8.5) {$8$};
\draw(27,-5)--(32,-5)--(32,-9);
\draw(27,-3)--(30, -3)--(30, -5);
\draw(32, -7)--(34, -7)--(34, -9);
\draw(30,-3)--(32, -3)--(32, -5);
\draw (28,-3)--(28,-1);
\draw(32, -3) --(32, -1);
\draw(34, -7)--(35, -7);
\end{tikzpicture}
\caption{The steps in the construction of $\rho(53417286)$.}
\label{fig:rho}
\end{figure}

The generic rectangulation $\gamma(x)$ is obtained from $\rho(x)$ by performing wall slides along the interior walls of $\rho(x)$.
An example is shown in Figure \ref{fig:gamma}.
For each interior wall $W$ of~$\rho(x)$, record a subsequence~$\sigma_W$ of $x$ consisting of the labels of rectangles adjacent to $W$.
We call~$\sigma_W$ the \emph{wall shuffle of $W$}.
For each wall $W$ of~$\rho(x)$, we temporarily label the vertices on $W$ using the rectangles adjacent to $W$ (as described below), and then use $\sigma_W$ and the labeling to determine which wall slides should be performed to obtain~$\gamma(x)$.
Every vertex on an interior wall $W$ is either the lower-right vertex or the upper-left vertex of some rectangle.
Note that no vertex of a diagonal rectagnulation is both the lower-right vertex of a rectangle and the upper-left vertex of a rectangle.
Thus the labeling described below will result in a total ordering of the entries of $\sigma_W$.
If the vertex is the lower-right vertex of some rectangle~$x_i$, then label the vertex with $x_i$.
Otherwise, the vertex is the upper-left vertex of some rectangle~$x_j$ and we label the vertex with $x_j$.
If~$W$ is a vertical wall, we perform wall slides so that the bottom to top order of the labeled vertices on~$W$ coincides with $\sigma_W$.
Since each vertical wall slide switches the order of a wall that extends to the \emph{left} of $W$ and a wall that extends to the \emph{right} of $W$, we explain why it is always possible to perform a sequence of wall slides so that the bottom to top order of the vertices agrees with $\sigma_W$.
Each vertex on $W$ that is the lower-right vertex of some rectangle is the endpoint of an edge extending to the left of $W$ and each vertex of~$W$ that is the upper-left vertex of some rectangle is the endpoint of an edge extending right of~$W$.
By the construction of~$\rho(x)$, the subsequence of~$\sigma_W$ consisting of the lower-right corner vertices on $W$ and the subsequence of~$\sigma_W$ consisting of upper-left corner vertices on~$W$ both agree with the bottom to top ordering of these vertices along~$W$, so it is possible to perform a sequence of wall slides to obtain the desired vertex order.
If~$W$ is a horizontal wall, we perform wall slides so that the left to right order of the labeled vertices on $W$ coincides with $\sigma_W.$
A similar argument shows that the desired vertex order can be obtained by some sequence of horizontal wall slides. 
Our definition of a wall shuffle differs from the definition given in \cite{grec}.  
The wall shuffles defined there can be obtained by deleting the first and last entries of the wall shuffles defined here.
This difference does not affect the definition of $\gamma$.

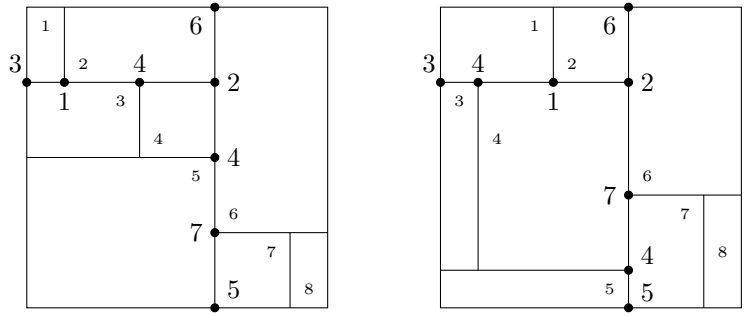
\begin{figure}
\begin{tikzpicture}[scale=.5]
\draw(18,-9) rectangle (26,-1);
\node at (18.5,-1.5) {\tiny{$1$}};
\node at (19.5,-2.5) {\tiny$2$};
\node at (20.5,-3.5) {\tiny$3$};
\node at (21.5,-4.5) {\tiny$4$};
\node at (22.5,-5.5) {\tiny$5$};
\node at (23.5,-6.5) {\tiny$6$};
\node at (24.5,-7.5) {\tiny$7$};
\node at (25.5,-8.5) {\tiny$8$};
\draw(18,-5)--(23,-5)--(23,-9);
\draw(18,-3)--(21, -3)--(21, -5);
\draw(23, -7)--(25, -7)--(25, -9);
\draw(21,-3)--(23, -3)--(23, -5);
\draw (19,-3)--(19,-1);
\draw(23, -3) --(23, -1);
\draw(25, -7)--(26, -7);
\filldraw(23, -9) circle [radius=3pt] ;
\filldraw(23, -7) circle [radius=3pt] ;
\filldraw(23, -5) circle [radius=3pt] ;
\filldraw(23, -3) circle [radius=3pt] ;
\filldraw(23, -1) circle [radius=3pt] ;
\node at (23.5, -8.5) {5};
\node at (22.5, -7) {$7$};
\node at (23.5, -5) {$4$};
\node at (23.5, -3) {$2$};
\node at (22.5, -1.5) {$6$};
\filldraw(18, -3) circle [radius=3pt];
\filldraw(19, -3) circle [radius=3pt];
\filldraw(21, -3) circle [radius=3pt];
\node at (17.7, -2.5) {$3$};
\node at (19, -3.5) {$1$};
\node at (21, -2.5) {$4$};

\draw(29,-9) rectangle (37,-1);
\node at (31.5,-1.5) {\tiny$1$};
\node at (32.5,-2.5) {\tiny$2$};
\node at (29.5,-3.5) {\tiny$3$};
\node at (30.5,-4.5) {\tiny$4$};
\node at (33.5,-8.5) {\tiny$5$};
\node at (34.5,-5.5) {\tiny$6$};
\node at (35.5,-6.5) {\tiny$7$};
\node at (36.5,-7.5) {\tiny$8$};
\draw(29,-8)--(34,-8)--(34,-9);
\draw(32, -3)--(32, -1);
\draw(29,-3)--(32, -3);
\draw(34, -6)--(36, -6)--(36, -9);
\draw(32,-3)--(34, -3);
\draw (30,-8)--(30,-3);
\draw(34, -8) --(34, -1);
\draw(36, -6)--(37, -6);
\filldraw(34, -9) circle [radius=3pt] ;
\filldraw(34, -8) circle [radius=3pt] ;
\filldraw(34, -6) circle [radius=3pt] ;
\filldraw(34, -3) circle [radius=3pt] ;
\filldraw(34, -1) circle [radius=3pt] ;
\node at (34.5, -8.6) {$5$};
\node at (34.5, -7.6) {$4$};
\node at (33.5, -6) {$7$};
\node at (34.5, -3) {$2$};
\node at (33.5, -1.5) {$6$};
\filldraw(29, -3) circle [radius=3pt];
\filldraw(30, -3) circle [radius=3pt];
\filldraw(32, -3) circle [radius=3pt];
\node at (28.7, -2.5) {$3$};
\node at (30, -2.5) {$4$};
\node at (32, -3.5) {$1$};
\end{tikzpicture}
\caption{The left diagram shows $\rho(53417286)$ (from Figure \ref{fig:rho}).  The right diagram shows $\gamma(53417286)$.}
\label{fig:gamma}
\end{figure}

To find $\gamma(x)$ from $\rho(x)$ in the example shown in Figure \ref{fig:gamma}, we consider the wall shuffle corresponding to every interior wall of $\rho(x).$
Since a wall slide cannot be performed along any wall with fewer than two rectangles adjacent to each side, we only need to examine the walls with at least two rectangles adjacent to each side.
There are two such walls in~$\rho(x)$.
First consider the vertical wall $W$ between rectangle 5 and rectangle 7.
For this wall, ${\sigma_{W}}=54726$.
We label the vertices along~$W$ as illustrated in the left diagram of Figure~\ref{fig:gamma}.  
In both diagrams of this figure, the rectangle labels have been reduced in size in order to emphasize the vertex labels.
To make the ordering of the labels along $W$ in $\gamma(x)$ agree with $\sigma_W$, the wall slide switching the order of vertices labeled 7 and 4 is performed.
Next consider the horizontal wall $W'$ of~$\rho(x)$ between rectangle 1 and rectangle 3.
For this wall, $\sigma_{W'}=3412$.
We again label the vertices along~$W'$ as illustrated in the left diagram of Figure \ref{fig:gamma}.  
Since the left to right order of these vertices in $\rho(x)$ is $3142$, we perform a wall side switching the order of the vertices labeled 1 and 4 along~$W'$ to obtain~$\gamma(x)$.

Additional examples of the map $\gamma$ are shown in Figures  \ref{fig:cover seq} and \ref{fig:concat}.
The generic rectangulations in both figures are labeled with $\gamma(x)$ where $x$ is the unique \mbox{2-clumped} permutation such that $\gamma(x)$ is the desired rectangulation.
In Figure \ref{fig:cover seq}, the bold entries in each permutation are inverted to find the next permutation in the sequence.  
Examining the permutation~$3142576$ associated with the leftmost generic rectangulation, we see that a generic pivot cannot be performed on the edge separating the shaded rectangles because the $3$ and $4$ are non-adjacent and there exists no permutation $x\in S_7$ such that the~$3$ and $4$ are adjacent in~$x$ and $\gamma(x)=\gamma(3142576)$. 

The theorem below is a rephrasing of a more general result from \cite[Section 2]{Reading}.\begin{theorem}\label{fiber}
Given a generic rectangulation $R$, the fiber $\gamma^{-1}(R)$ forms an interval in the right weak order.  
\end{theorem}

%
%

Using the construction in the proof of \cite[Proposition 4.2]{grec}, we define $\psi$, the inverse of the restriction of $\gamma$ to the set of 2-clumped permutations.
To demonstrate that~$\gamma$ is a surjective map, that proof begins with an arbitrary generic rectangulation~$R$ and the associated diagonal rectangulation $D$. 
A permutation $x$ is constructed, entry by entry, so that $\rho(x)=D$ and each wall shuffle of $R$ is a subsequence of~$x$.
Let $T_{i-1}$ be the partial diagonal rectangulation obtained after completing the first $i-1$ steps in the construction of~$\rho(x)$.
In the proof of \cite[Proposition 4.2]{grec}, the requirement that $\rho(x)=D$ is translated into the requirement that (in $D$) the left side and bottom of rectangle $x_{i}$ are contained in~$T_{i-1}$ for all $i\in[n]$. 
We say that $x_1 \cdots x_i$ \emph{respects the wall shuffles of $R$} if there exists no $x_j\in [n]-\{x_1,...,x_i\}$ such that $x_j$ precedes some element of $\{x_1,...,x_i\}$ in a wall shuffle of $R$.
The requirement that each wall shuffle of $R$ is a subsequence of $x$ is equivalent to the requirement that $x_1\cdots x_i$ respects the wall shuffles of $R$ for all $i \in [n]$. 
Using these equivalences, to show that $\gamma$ is surjective, the proof of \cite[Proposition 4.2]{grec} demonstrates that for all $i\in [n]$ there exists some $x_i \notin \{x_1,...,x_{i-1}\}$ such that the left side and bottom of rectangle $x_i$ are contained in $T_{i-1}$ and $x_1\cdots x_i$ respects the wall shuffles of $R$.
In this construction, each time an entry of $x$ is selected, there may be a choice.
We define $\psi(R)$ be the permutation obtained by choosing the minimum possible entry at each step.
We will prove the following proposition.

\begin{proposition}\label{prop:psi}
The map $\psi: gRec_n \to Cl^2_n$ is the inverse of the restriction of $\gamma$ to 2-clumped permutations.
\end{proposition}

To prove Proposition \ref{prop:psi}, we will use the following proposition, which appears as part of \cite[Proposition 2.2]{grec}.
\begin{proposition}
\label{prop:min2clumped}
A permutation $y$ is the minimal element of the right weak order such that $\gamma(y)=R$ if and only if $y$ is a 2-clumped permutation.
\end{proposition}

\begin{proof}[Proof of Proposition \ref{prop:psi}]
Let $R\in gRec_n$.
To prove the proposition, it suffices to demonstrate that $\psi(R)\in Cl^2_n$, or equivalently, by Proposition \ref{prop:min2clumped}, that $\psi(R)$ is the minimal element of the right weak order mapping to $R$ under $\gamma$.
Let $\psi(R)=p=p_1\cdots p_n$ and $x\in S_n$ such that $x \lessdot p$ in the right weak order.
Then there exists some $i\in [n-1]$ such that $x=p_1\cdots p_{i-1}p_{i+1} p_i p_{i+2}\cdots p_n$ and $p_{i+1}<p_i$.
Since $x_j=p_j$ for all $j\in [i-1]$, and $p_i$ is the smallest entry of any permutation starting with $p_1\cdots p_{i-1}$ and mapping to $R$ under $\gamma$, we have that $\gamma(x)\neq R$.
By Theorem~\ref{fiber}, the permutation $p$ is the minimal element of the right weak order such that~$\gamma(p)=R$. 
\end{proof}

%
%

\section{The lattice of generic rectangulations}
\label{weak order}

In this section, we prove Theorem \ref{thm: covers}.  
To do so, we rely on results about diagonal rectangulations from \cite{drec} and results about generic rectangulations from~\cite{grec}.
Recall that we call each element of $Cl_n^1$ a twisted Baxter permutation.
The map $\rho: S_n \to dRec_n$ restricts to a bijection between $Cl^1_n$  and $dRec_n$~\cite[Theorem 6.1]{drec}.
The right weak order on $S_n$ modulo the fibers of $\rho$ is a lattice on the set of twisted Baxter permutations.
Applying $\rho$ to the elements of this lattice results in a lattice of diagonal rectangulations of size $n$ which, reusing notation, we call $dRec_n$.

To describe the cover relations of $dRec_n$, we define diagonal pivots.  
Let $D$ and~$D'$ be diagonal rectangulations.
Diagonal pivots and generic pivots are closely related.
Specifically,~$D$ and~$D'$ are related by a \emph{diagonal pivot} if and only if they are related by a local change shown in one of the left three diagrams of  Figure \ref{fig:moves}, where the dotted segment of each diagram is ignored.
In this paper, we call each of these local moves diagonal pivots to emphasize that they are performed on diagonal rectangulations rather than generic rectangulations.
The reader should note that this differs from the definition of a diagonal pivot given in \cite{drec}, where the move illustrated in the leftmost diagram of Figure \ref{fig:moves} is called a diagonal pivot and the other two local moves which we also call diagonal pivots are instead called vertex pivots.
The cover relations in $dRec_n$ are described in \cite[Theorem 7.1]{drec}:

\begin{theorem}
\label{thm:coverindrec}
Two diagonal rectangulations $D$ and $D'$ of size $n$ have $D \lessdot D'$ in~$dRec_n$ if and only if they are related by a diagonal pivot such that the pivoted edge is vertical in $D$.
\end{theorem}
 
The following is a restatement of \cite[Theorem 4.5, part (3)]{grec}:

\begin{theorem}
\label{thm:drecmove}
Assume that $x\lessdot y$ in the right weak order.  Then $\rho(x)=\rho(y)$ if and only if~$x$ is obtained from $y$ by either a $(3412 \to 3142)$ move or a $(2413 \to 2143)$ move.
\end{theorem}

The analogous result also holds for generic rectangulations \cite[Proposition 4.3]{grec}.

\begin{theorem}\label{thm:grecmove}
Assume that $x\lessdot y$ in the right weak order.  Then $\gamma(x)=\gamma(y)$ if and only if~$x$ is obtained from $y$ by a $(35124 \to 31524)$ move, a $(35142 \to 31542)$ move, a $(24513\to 24153)$ move, or a $(42513 \to 42153)$ move.
\end{theorem}

The following is a specialization of a more general result \cite[Proposition 2.2]{Reading} to the case of 2-clumped permutations and generic rectangulations.

\begin{proposition}\label{covers}
Let $y\in Cl_n^2$.  
Then $R\in gRec_n$ is covered by $\gamma(y)$ in the lattice of generic rectangulations of size $n$ if and only if there exists some permutation $x\in S_n$ with $\gamma(x)=R$ such that $x\lessdot y$ in the right weak order on $S_n$.
\end{proposition}

The following proposition is a specialization of \cite[Prop 9-5.4]{new}:

\begin{proposition} \label{prop:quotient}

Given distinct $R_1, R_2 \in gRec_n$, we have that $R_1 \lessdot R_2$ in $gRec_n$ if and only if there exist $x_1, x_2 \in S_n$ such that $\gamma(x_1)=R_1$ and $\gamma(x_2)=R_2$ with $x_1 \lessdot x_2$ in the right weak order on $S_n$. 
\end{proposition}


%

In light of Proposition \ref{covers}, the next proposition is one direction of Theorem \ref{thm: covers}.


\begin{proposition}\label{prop:part1}
Let $x\in S_n$ and $y\in Cl_n^2$ such that $x\lessdot y$ in the right weak order.  Then $\gamma(x)=R_1$ and $\gamma(y)=R_2$ are related by a generic pivot or wall slide shown in Figure \ref{fig:moves} with the bottom diagram corresponding to $R_1$ and the top diagram corresponding to $R_2$.
\end{proposition}

\begin{proof}
Let $y=y_1\cdots y_n$.
Since $x\lessdot y$ in the right weak order on $S_n$, we have that $x=y_1\cdots y_{i-1}y_{i+1}y_{i}y_{i+2} \cdots y_n$ with~$y_{i+1}<y_{i}$.  
By Proposition \ref{prop:min2clumped}, $\gamma(x)\neq \gamma(y)$.
We consider two cases: $\rho(x)=\rho(y)$ and $\rho(x)\neq \rho(y)$.  

First assume that $\rho(x)=\rho(y)$.  
Since $\gamma(x)\neq \gamma(y)$, rectangulations $R_1$ and $R_2$ differ by wall slides.
Every wall shuffle of $R_1$ is a subsequence of $x$, so interchanging two elements of $x$ to obtain~$y$ changes the order of at most two elements of any wall shuffle.  
Suppose first that more than one wall shuffle of $R_1$ differs from 
the corresponding wall shuffle of $R_2$.
Specifically, assume that the adjacent pair $y_{i+1}y_i$ appears in two or more wall shuffles of $R_1$, so rectangles~$y_{i+1}$ and $y_i$ are adjacent to at least two shared walls.
Since $\rho(x)=\rho(y)$, the corresponding wall shuffles of $R_2$ contain adjacent pair $y_iy_{i+1}$, implying that rectangles $y_i$ and $y_{i+1}$ are on opposite sides of those walls.
Two rectangles can be adjacent to opposite sides of at most one vertical wall and at most one horizontal wall.
If these simultaneously occur, then the rectangles are part of a group of four rectangles that share a single vertex, contradicting the assumption that~$R_1$ is a generic rectangulation.  
Thus rectangles~$y_i$ and $y_{i+1}$ share a single wall, implying that exactly one wall shuffle of $R_1$ differs from the corresponding wall shuffle of $R_2$.
If the shared wall $W$ is horizontal, then since $y_{i+1}<y_i$ and the label of each rectangle above $W$ is smaller than the label of each rectangle below $W$, rectangle $y_{i+1}$ is above $W$ and rectangle~$y_i$ is below~$W$.  
Switching their order in $x$ to obtain $y$ results in the horizontal wall slide shown in the far right diagram of Figure \ref{fig:moves}.
Similarly, if the shared wall $W$ is vertical, since the label of each rectangle to the left of $W$ is smaller than the label of each rectangle to the right of $W$, rectangle $y_{i+1}$ is left of $W$ and rectangle $y_i$ is right of $W$.  
Switching their order results in the vertical wall slide illustrated in Figure \ref{fig:moves}.

Now assume that $\rho(x)\neq\rho(y)$.  
Theorem \ref{thm:coverindrec} implies that $\rho(x)$ and $\rho(y)$ are related by a diagonal pivot such that the pivoted edge is vertical in $\rho(x)$.
If there exist $a,b$ with $y_{i+1}<a,b<y_i$ such that $a$ occurs to the left of position $i$ in $y$ and $b$ occurs to the right of position $i+1$ in $y$, 
then~$x$ is obtained from $y$ by a $(2413\to 2143)$-move or a $(3412\to 3142)$-move.
 By Theorem \ref{thm:drecmove}, this implies that $\rho(x)=\rho(y)$, contradicting our initial assumption, so this cannot occur.
We now consider three remaining cases.  
In this proof, it will be convenient to use the correspondences established in the proof of \cite[Theorem 7.1]{drec} between each case and a specific diagonal pivot.\\
Case 1: $y_i = y_{i+1}+1$.  
In this case, $\rho(x)$ and~$\rho(y)$ are related by the \emph{diagonal} pivot shown in the leftmost diagram of Figure \ref{fig:moves}.
We will show that this implies that $R_1$ and $R_2$ are related by the \emph{generic} pivot shown in the leftmost diagram of Figure~\ref{fig:moves}.
Let $W_1$ denote a wall of~$\rho(x)$ (or equivalently a wall of  $R_1$) that is adjacent to neither rectangle $y_i$ nor rectangle~$y_{i+1}$ and~$W_2$ the corresponding wall of $\rho(y)$ (or  equivalently $R_2$).
Since $W_2$ is also not adjacent to either rectangle, $\sigma_{W_1}=\sigma_{W_2}$.  
Thus the wall shuffles of $R_1$ and $R_2$ differ only on walls adjacent to the union of rectangles $y_i$ and $y_{i+1}$. 
We consider each of the wall shuffles of $R_1$ containing $y_i$ or~$y_{i+1}$.    
In $R_1$, the wall shuffle associated to the pivoted edge is~$y_{i+1}y_i$ and in~$R_2$ it is~$y_iy_{i+1}$.
Now examine $\sigma_{W_{1b}}, \sigma_{W_{1a}}, \sigma_{W_{1l}},$ and $\sigma_{W_{1r}}$,  the wall shuffles of the walls below, above, to the left, and to the right of the union of rectangles~$y_i$ and $y_{i+1}$ in $R_1$.  
We compare these wall shuffles with $\sigma_{W_{2b}},\sigma_{W_{2a}}, \sigma_{W_{2l}},$ and~$\sigma_{W_{2r}}$, the corresponding wall shuffles in $R_2$, to demonstrate that they differ exactly as shown in Figure~\ref{fig:moves}.
Since $\sigma_{W_{1b}}$ and $\sigma_{W_{1a}}$ contain both $y_i$ and $y_{i+1}$, and since~$y_i$ and $y_{i+1}$ are adjacent in~$x$, they are also adjacent in~$\sigma_{W_{1b}}$ and $\sigma_{W_{1a}}$.
The wall shuffle of a horizontal wall records the ordering of the left edges of rectangles below the wall and the right edges of rectangles above the wall so the adjacency of $y_{i+1}$ and $y_i$ in these wall shuffles implies that no edge of $R_1$ is adjacent to the interior of the bottom of rectangle~$y_i$ or the top of rectangle~$y_{i+1}$.
Similarly, in $R_2$ no edge is adjacent to the interior of the left side of rectangle~$y_{i+1}$ or the right side of rectangle~$y_i$.
Since rectangle $y_{i+1}$ is not adjacent to~$W_{2b}$ and only $y_i$ and~$y_{i+1}$ are switched in~$y$, wall shuffle~$\sigma_{W_{2b}}$ is obtained by removing $y_{i+1}$ from~$\sigma_{W_{2b}}$.
Using the same argument, we see that: wall shuffle~$\sigma_{W_{1a}}$ is obtained by removing $y_i$ from~$\sigma_{W_{1a}}$, wall shuffle $\sigma_{W_{2l}}$ is obtained by inserting~$y_i$ immediately before $y_{i+1}$ in $\sigma_{W_{1l}}$, and wall shuffle~$\sigma_{W_{2r}}$ is obtained by inserting~$y_{i+1}$ immediately after~$y_i$ in $\sigma_{W_{1r}}$.
Thus the wall shuffles of $R_1$ and $R_2$ differ exactly as shown in the leftmost diagram of Figure \ref{fig:moves} and no walls are adjacent to the interior of any dashed segment.\\
Case 2:  $y_i>y_{i+1}+1$ and every $a$ with $y_{i+1}<a<y_i$ occurs to the right of position $i+1$ in~$y$.
In this case, $\rho(x)$ and $\rho(y)$ are related by the \emph{diagonal} pivot shown in the second diagram of Figure \ref{fig:moves}.
Now consider $R_1$ and $R_2$.
As in Case~1, only wall shuffles containing $y_i$ or~$y_{i+1}$ are effected by interchanging $y_i$ and $y_{i+1}$ in $x$ to obtain $y$.  
Again, 
examining each wall shuffle of~$R_1$ and relating it to the corresponding wall shuffle of $R_2$, we see that $R_1$ and $R_2$ are related as shown in the second diagram of Figure \ref{fig:moves}.\\
Case 3:  $y_i>y_{i+1}+1$ and every $a$ with $y_{i+1}<a<y_i$ occurs to the left of position~$i$ in $y$.
In this case, $\rho(x)$ and $\rho(y)$ are related by the \emph{diagonal} pivot shown in the third diagram of Figure \ref{fig:moves}, and this case is handled like Case 2.
\end{proof}

The other direction of Theorem \ref{thm: covers} follows from the following sequence of propositions.

\begin{proposition}\label{prop:pattern}
Let $y\in S_n$ and $y_iy_jy_ky_ly_m$ be an occurrence of the pattern \mbox{$3$-$5$-$1$-$4$-$2$} in $y$.  
If every $y_p$ satisfying $y_m<y_p<y_l$ occurs before $y_j$ in $y$, then $y$ contains an occurrence of the pattern $3$-$51$-$4$-$2.$
\end{proposition}

\begin{proof}
Since every $y_p$ such that $y_m<y_p<y_l$ occurs before $y_j$ in $y$,
each entry of~$y$ between $y_j$ and $y_k$ is either greater than $y_l$ or less than $y_m$.  
If every entry of~$y$ between $y_j$ and $y_k$ is greater than $y_l$, then the subsequence $y_iy_{k-1}y_ky_ly_m$ is an occurrence of the pattern $3$-$51$-$4$-$2$ in $y$.
Otherwise, let $y_q$ denote the first entry of $y$ between $y_j$ and $y_k$ such that $y_q<y_m$.
In this case, $y_iy_{q-1}y_qy_ly_m$ is an occurrence of the pattern $3$-$51$-$4$-$2$.
\end{proof}

\begin{proposition} \label{prop:gpiv1}
Let $R_1, R_2\in gRec_n$  such that $R_1$ and $R_2$ are related by a single generic pivot as shown in the leftmost diagram of Figure \ref{fig:moves} 
 with the lower illustration corresponding to $R_1$ and the upper illustration corresponding to $R_2$.
Then $R_1 \lessdot R_2$ in~$gRec_n$.
\end{proposition}

\begin{proof}  Let $R_1$ and $R_2$ be generic rectangulations as described in the proposition and $E$ be the horizontal edge of $R_2$ that is pivoted to form $R_1$. 
Let $\psi(R_2)=y=y_1 \cdots y_n$, the unique element of $Cl_n^2$ such that $\gamma(y)=R_2$.
Label the rectangles directly below and above $E$ rectangle $y_i$ and rectangle $y_j$ respectively.  
 
As defined in the description of $\gamma$, let $T_i$ be the partial diagonal rectangulation obtained after the first $i$ steps in the construction of $\rho(y)=D_2$.
By the definition of $\psi(R_2)$, entry $y_{i+1}$ is the smallest element of~$\{y_{i+1},...,y_n\}$ such that the left side and bottom of rectangle $y_{i+1}$ are contained in $T_i$ and $y_1 \cdots y_{i+1}$ respects the wall shuffles of~$R_2$.
To show that~$R_1 \lessdot R_2$ in~$gRec_n$, we first demonstrate that $i+1=j$. 

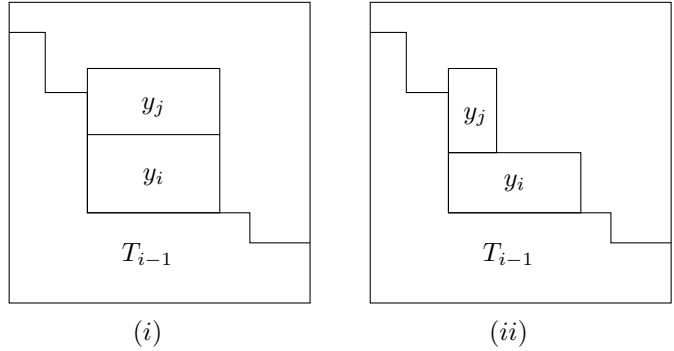
\begin{figure}
\begin{tikzpicture}[scale=.8]
\coordinate (A) at (0,0);
\coordinate (C) at (5,5);
\coordinate (B) at (1.3, 1.5);
\coordinate (D) at (3.5, 3.9);
\draw (A) rectangle (C);
\draw (0,4.5)--(.6,4.5)--(.6,3.5)--(1.3,3.5)--(1.3,1.5)--(4,1.5)--(4,1)--(5,1);
\draw (B) rectangle (D);
\draw (1.3, 2.8)--(3.5, 2.8);
\node at (2.4, 2.1) {$y_i$};
\node at (2.4, 3.25) {$y_j$};
\node at (2.3, .8) {$T_{i-1}$};

\node at (2.3, -.5) { $(i)$};

\coordinate (A) at (6,0);
\coordinate (C) at (11,5);
\coordinate (B) at (7.3, 1.5);
\coordinate (D) at (9.5, 3.9);
\draw (A) rectangle (C);
\draw (6,4.5)--(6.6,4.5)--(6.6,3.5)--(7.3,3.5)--(7.3,1.5)--(10,1.5)--(10,1)--(11,1);
\draw (7.3, 1.5) rectangle (9.5, 2.5); 
\draw (7.3, 2.5) rectangle (8.1, 3.9);
\node at (8.4, 2) {$y_i$};
\node at (7.75, 3.1) {$y_j$};
\node at (8.3, .8) {$T_{i-1}$};

\node at (8.3, -.5) {$(ii)$};

%

\end{tikzpicture}
\caption{Illustrations for the proofs of Propositions \ref{prop:gpiv1}-\ref{prop:ii}.}
\label{fig:D_2}
\end{figure}

To reach that goal, we will begin by showing that the bottom and left side of rectangle $y_j$ are contained in~$T_i$.  
Diagram (i) of Figure \ref{fig:D_2} illustrates a possible configuration of rectangles~$y_i$ and $y_j$ with respect to $T_{i-1}$ in $D_2$.
Since $R_2$ is a \emph{generic} rectangulation, the wall containing~$E$ is~$E$ itself.  
Rectangulations $R_2$ and~$D_2$ differ only by a sequence of wall slides and no wall slides can be performed along~$E$ so the top of rectangle $y_i$ and the bottom of rectangle $y_j$ coincide in~$D_2$.
Thus the bottom of rectangle $y_j$ is contained in $T_i$. 
To demonstrate that the left edge of rectangle $y_j$ is contained in $T_i$, assume for a contradiction that this is not the case (as illustrated in Diagram (i) of Figure \ref{fig:D_2}).  
Then there exists some rectangle~$y_p$ not contained in $T_i$, such that the right side of rectangle $y_p$ is adjacent to the left side of rectangle $y_j$ and the bottom of rectangle $y_p$ is contained in $T_i$.  
In~$y$, the entry~$y_p$ occurs after $y_i$ but before $y_j$.  
However, after wall slides are performed to obtain~$\gamma(y)=R_2$ from~$D_2$, this implies that the lower-right corner of rectangle~$y_p$ is contained in the interior of the left side of rectangle $y_j$, contradicting the assumption that~$E$ is a pivotable edge in~$R_2$.  

Now we show that adding rectangle $y_j$ to the partial rectangulation immediately after $y_i$ respects the wall shuffles of $R_2$.  
Let $W_l, W_r, W_b$ and $W_a$ be the walls respectively to the left of, to the right of, below, and above rectangle $y_j$ in Diagram~(i) of Figure \ref{fig:D_2}. 
Since only rectangles~$y_i$ and~$y_j$ border $W_b$, following $y_i$ immediately by~$y_j$ in $y$ respects this wall shuffle.
If there is some~$y_p$ between $y_i$ and $y_j$ in ${\sigma_{W_l}}$, then rectangle $y_p$ is on the left side of $W_l$ and in~$R_2$ the bottom right vertex of rectangle~$y_p$ is contained in the interior of the left side of rectangle~$y_j$, contradicting the assumption that  $E$ is pivotable.
The analogous argument shows that $y_i$ and~$y_j$ are adjacent in $\sigma_{W_r}$.  
Now consider $W_a$.
If rectangle $y_j$ is the lower leftmost rectangle on~$W_a$, then $y_j$ is the first entry of $\sigma_{W_a}$ so following $y_i$ immediately by~$y_j$ in~$y$ respects the wall shuffle of $W_a$.  
Otherwise, the upper-left vertex of rectangle~$y_j$ coincides with a vertex of $T_{i-1}$.
This case is illustrated in Figure~\ref{fig:W_a}.
Let rectangle $y_l$ be the rectangle contained in $T_{i-1}$ whose upper-right vertex coincides with the upper-left vertex of rectangle~$y_j$, let rectangle $y_p$ be the leftmost rectangle not contained in~$T_{i-1}$ such that the bottom of rectangle $y_p$ is contained in $W_a$, and let rectangle~$y_k$ be the rightmost rectangle such that the bottom of rectangle $y_k$ is contained in $W_a$.  
If following $y_i$ immediately by $y_j$ does not respect $\sigma_{W_a}$, then~$y_p$ precedes $y_j$ in $y$.
Since rectangle $y_k$ is the final rectangle above and adjacent to~$W_a$, entry $y_k$ follows~$y_j$ in~$y$.
Note that $y_p<y_k<y_l<y_j<y_i$ so $y_ly_iy_py_jy_k$ is an occurrence of the pattern $3$-$5$-$1$-$4$-$2$ in $y$.
Every rectangle $y_q$ with label satisfying $y_k<y_q<y_j$ is in $T_{i-1}$ because the label $y_q$ is on the diagonal of the square $S$ between labels $y_k$ and $y_j$.
Thus every such $y_q$ precedes $y_i$ in $y$.
By Proposition \ref{prop:pattern}, permutation $y$ contains a $3$-$51$-$4$-$2$ pattern, contradicting the assumption that $y=\psi(R_2)$.
 
\begin{figure}
\begin{tikzpicture}[scale=1.5]
\draw(0,6)--(1,6)--(1,5)--(3,5)--(3,4)--(5,4);
\node at (1.2,4.4) {$T_{i-1}$};
\draw(3,4.5)--(4.5,4.5)--(4.5,4);
\draw(3,5)--(4.5, 5)--(4.5, 4.5);
\node at (3.75, 4.25) {$y_i$};
\node at (3.75, 4.75) {$y_j$};
\draw(2.2, 5)--(2.2, 5.5);
\node at (2.4, 5.2) {$y_k$};
\draw(1,5.5)--(1.7, 5.5)--(1.7, 5);
\node at (1.35, 5.25) {$y_p$};
\node at (2.8, 4.8) {$y_l$};
\end{tikzpicture}
\caption{An illustration used in the proof of Proposition \ref{prop:gpiv1}}\label{fig:W_a}
\end{figure}
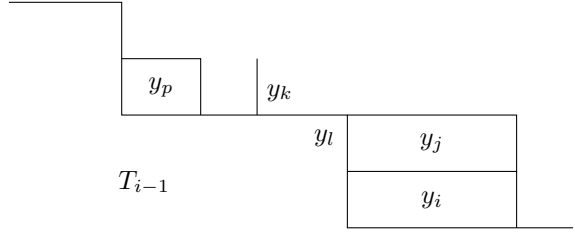
 
We have shown that the bottom and left side of rectangle $y_j$ are contained in~$T_i$ and adding rectangle $y_j$ to $T_i$ immediately after rectangle $y_i$ respects the wall shuffles of $R_2$. 
Next we demonstrate that $y_j$ is the smallest of $\{y_{i+1},...,y_n\}$ with these properties.  
Assume that there is some $y_p \in \{y_{i+1},...,y_n\}$ with these properties such that $y_p<y_j$ in numerical order.
As demonstrated in the previous paragraph, rectangle $y_p$ is not adjacent to $W_l$, the left wall of rectangle $y_j$ which is also the left wall of rectangle $y_i$.
Since $y_p<y_j$ in numerical order, rectangle $y_p$ contains a label above and to the left of the label for rectangle~$y_j$ so rectangle $y_p$ shares no walls with rectangle $y_i$.
Since the addition of rectangle $y_p$ to the partial rectangulation after rectangle $y_i$ respects the wall slides of $R_2$, this implies that the addition of rectangle~$y_p$ to the partial rectangulation immediately before rectangle~$y_i$ also respects the wall slides of $R_2$.
Because $y_p<y_j$ and the left and bottom sides of rectangle~$y_p$ are contained in $T_i$, the left and bottom sides of rectangle $y_p$ are also contained in $T_{i-1}$.
However, since $y_p<y_i$, this contradicts our choice of $y_i$ as the $i$th entry of $\psi(R_2)$, i.e. rectangle $y_p$ could have been added to~$T_{i-1}$ instead of rectangle~$y_i$.
Thus $j=i+1$.  
Observing that $\gamma(y_1 \cdots y_{j} y_i \cdots y_n)=R_1$ completes the proof. 
\end{proof}

\begin{proposition}\label{prop:ii}
Let $R_1, R_2\in gRec_n$ such that $R_1$ and $R_2$ are related by a single generic pivot as shown in the second diagram from the left in Figure \ref{fig:moves}
with the lower illustration corresponding to $R_1$ and the upper illustration corresponding to~$R_2$.
Then $R_1 \lessdot R_2$ in~$gRec_n$.
\end{proposition}

\begin{proof}
Let $D_2$ denote the diagonal rectangulation associated with $R_2$.
 As in the proof of Proposition \ref{prop:gpiv1}, let $E$ denote the horizontal edge of $R_2$ that is pivoted to form $R_1$ and let permutation $y=y_1 \cdots y_n=\psi(R_2)$.
Label the rectangle directly below $E$ with $y_i$ and the rectangle directly above $E$ with $y_j$.
Let $W_l, W_r, W_b,$ and $W_a$ refer to the walls respectively to the left of, to the right of, below, and above rectangle $y_j$ in $D_2$.
  As in the proof of Proposition \ref{prop:gpiv1}, we demonstrate that $i+1=j$. 

Diagram (ii) of Figure \ref{fig:D_2} shows a possible configuration of rectangles $y_i$ and $y_j$ with respect to $T_{i-1}$ in $D_2$.
In $D_2$, as in $R_2$, the upper-left vertex of rectangle $y_i$ and the lower left vertex of rectangle $y_j$ coincide.
Additionally, as in $R_2$, the lower-right vertex of rectangle~$y_j$ is contained in the interior of the top of rectangle $y_i$ in $D_2$.   
To see why the second statement is true, note that performing a wall slide to switch the relative locations of the lower-right vertex of rectangle $y_j$ and the upper-right vertex of rectangle $y_i$ results in a rectangulation which is not diagonal. 
Thus the bottom of rectangle $y_j$ is contained in $T_i$.
Arguments identical to those used in the proof of Proposition \ref{prop:gpiv1} show that the left edge of rectangle~$y_j$ is also contained in $T_i$ and that adding rectangle $y_j$ immediately following rectangle $y_i$ respects~$\sigma_{W_l}$ and $\sigma_{W_a}$. 
Since rectangle $y_j$ is the lowermost rectangle on the left side of $W_r$, the wall shuffle~$\sigma_{W_r}$ begins with $y_j$.
If $y_j$ does not immediately follow $y_i$ in $\sigma_{W_b}$, then rectangles $y_i$ and $y_j$ are not in the configuration shown in the second diagram of Figure \ref{fig:moves}.
Specifically, if~$y_j$ does not immediately follow $y_i$ in $\sigma_{W_b}$, then there exists some rectangle $y_p$ whose left side is adjacent to rectangle $y_i$ and whose top is contained in $W_b$.
Performing wall slides to obtain~$R_2$ from $D_2$, the lower-right corner of rectangle $y_j$ is not contained in the interior of the top of rectangle $y_i$.
Thus $y_1 \cdots y_i y_j$ respects the wall shuffles of $R_2$.
Again using the argument from the proof of Proposition \ref{prop:gpiv1}, we see that $y_j$ is the smallest element of $\{y_{i+1},...,y_n\}$ such that the walls of the corresponding rectangle are contained in $T_i$ and whose selection respects the wall shuffles of $R_2$ so $i+1=j$.
The proof is completed by observing that $\gamma(y_1\cdots y_{i+1} y_i \cdots y_n)=R_1$.
\end{proof}

We now describe four maps that will be used to complete the proofs of Theorem~\ref{thm: covers} and Theorem~\ref{thm:coproduct}.
Let $\rf$ be the automorphism of generic rectangulations of size $n$ that takes a generic rectangulation $R$ to the generic rectangulation $R'$ obtained by reflecting $R$ about the upper-left to lower-right diagonal of the square $S$.
Let $\rfu$ be the automorphism that takes a generic rectangulation $R$ to the generic rectangulation $R'$ obtained by reflecting $R$ about the lower-left to upper-right diagonal of $S$.
Let $\rp: S_n \to S_n$ denote the map on permutations that reverses the positions of entries in the one-line notation for a permutation.
Let $\rv: S_n \to S_n$ denote the map on permutations that reverses the values of the permutation, replacing each entry $x_i$ of the permutation $x$ with $n+1-x_i$.
For example, $\rp(34521)=12543$ and $\rv(34521)=32145$.

The maps $\rp$ and $\rv$ are antiautomorphisms of the right weak order on $S_n$.
As noted in \cite[Remark 6.5, Remark 6.10]{drec}, $\rf \circ \rho = \rho \circ \rp$ and $\rfu \circ \rho = \rho \circ \rv$.
Since applying $\rf$ to an arbitrary generic rectangulation reverses each wall shuffle, the wall shuffles of  rectangulation $\rf \circ \gamma $ agree with the wall shuffles of rectangulation~\mbox{$\gamma \circ \rp$}.
Thus $\rf \circ \gamma = \gamma \circ \rp$.
Additionally, given a generic rectangulation $R$ with wall shuffle $\sigma_W= x_{i_1} \cdots x_{i_p}$, the wall shuffle of the corresponding wall $W'$ in $\rfu(R)=R'$ is $\sigma_{W'}=(n+1-x_{i_1}) \cdots (n+1-x_{i_p})$.
Thus the wall shuffles of $\rfu \circ  \gamma$ agree with the wall shuffles of $\gamma \circ \rv$ so $\rfu \circ \gamma = \gamma \circ \rv$.

\begin{lemma}\label{lemma:rf}
The map $\rf$ is an antiautomorphism of the lattice of generic rectangulations.
\end{lemma}

\begin{proof}
Let $R_1, R_2\in gRec_n$ such that $R_1\lessdot R_2$ in $gRec_n$.  
By Proposition \ref{prop:quotient}, there exist $x_1, x_2\in S_n$ such that $\gamma(x_1)=R_1$,  $\gamma(x_2)=R_2$, and $x_1\lessdot x_2$ in the right weak order on~$S_n$.
Because $\rp$ is an antiautomorphism of the right weak order, we have that $\rp(x_1)\gtrdot \rp(x_2)$ in the right weak order.  
Since $\gamma(x_1)\neq\gamma(x_2)$, and $\rf\circ\gamma=\gamma\circ\rp$, we have that 
 \mbox{$\gamma(\rp(x_1))\neq \gamma(\rp(x_2))$.}
Again applying Proposition~\ref{prop:quotient}, we obtain $\gamma(\rp(x_1))\gtrdot \gamma(\rp(x_2))$ in $gRec_n$.
Since $\gamma\circ \rp=\rf\circ\gamma$, we conclude that $\rf(R_1)\gtrdot \rf(R_2)$ in $gRec_n$.
An identical argument shows that if $\rf(R_1)\gtrdot \rf(R_2)$ in $gRec_n$, then $R_1\lessdot  R_2$ in~$gRec_n$.
\end{proof}

\begin{proposition}\label{prop:iii}
Let $R_1, R_2\in gRec_n$ such that $R_1$ and $R_2$ are related by a single generic pivot as shown in the center diagram of Figure \ref{fig:moves}
with the lower illustration corresponding to~$R_1$ and the upper illustration corresponding to $R_2$. 
Then $R_1 \lessdot R_2$ in~$gRec_n$.
\end{proposition}

\begin{proof}
Let $R_1$ and $R_2$ be generic rectangulations as described in the proposition.
Generic rectangulations $\rf(R_1)$ and $\rf(R_2)$ meet the conditions described in Proposition \ref{prop:ii} (with the lower diagram of Figure \ref{fig:moves} corresponding to $R_2$ and the upper diagram of Figure \ref{fig:moves} corresponding to $R_1$) so $\rf(R_1) \gtrdot \rf(R_2)$ in $gRec_n$.
Thus by Lemma \ref{lemma:rf}, $R_1 \lessdot R_2$ in~$gRec_n$.
\end{proof}

\begin{proposition}
\label{prop:wallslide}
Let $R_1, R_2\in gRec_n$ such that $R_1$ and $R_2$ are related by a single wall slide as shown in the fourth or fifth diagram of Figure~\ref{fig:moves}
with the lower illustration corresponding to $R_1$ and the upper illustration corresponding to $R_2$. 
Then $R_1 \lessdot R_2$ in~$gRec_n$.
\end{proposition}

\begin{proof}
First assume that $R_1$ and $R_2$ differ by a single vertical wall slide as shown in the fourth diagram of Figure \ref{fig:moves}.  
Let $W_1$ and $W_2$ respectively denote the walls in $R_1$ and $R_2$ on which the wall slide occurs. 
Let $\psi(R_2)=y=y_1 \cdots y_n$. 
We wish to find some $j$ such that interchanging $y_j$ and $y_{j+1}$ in $y$ results in a permutation $x$ with $\gamma(x)=R_1$.
Let $\sigma_{W_2}=y_{w_1}\cdots y_{w_i}y_{w_{i+1}}\cdots y_{w_f}$ be the wall shuffle of~$W_2$ and $\sigma_{W_1}=y_{w_1}\cdots y_{w_{i+1}}y_{w_{i}}\cdots y_{w_f}$ be the wall shuffle of $W_1$ as illustrated in Figure~\ref{fig:moves w/ labels}.
To prove that $R_1 \lessdot R_2$ in $gRec_n$, we will show that $y_{w_i}$ and $y_{w_{i+1}}$ are adjacent in~$y$ and that switching their locations in~$y$ results in a permutation $x$ such that $\gamma(x)=R_1$.
Using the definition of the map $\rho$, we observe that $y_{w_{i+1}}<y_{w_1}<y_{w_1}+1=y_{w_f}<y_{w_i}$.
Let~$a_1 \cdots a_l$ be the sequence of elements between~$y_{w_i}$ and~$y_{w_{i+1}}$ in~$y$.
Let $a_m$ be the last element of the sequence satisfying $y_{w_{i+1}}<a_m<y_{w_{1}}$, if such an entry exists.
If rectangle~$a_m$ were not adjacent to $W$, then by the definition of $\rho(y)$, rectangle $y_{w_{i+1}}$ would also not be adjacent to $W$. 
Thus, rectangle $a_m$ must be adjacent to $W$.
However, this implies that~$a_m$ occurs between $y_{w_{i}}$ and $y_{w_{i+1}}$ in $\sigma_{W_2}$, a contradiction. 
Now let $a_m$ be the first element of the sequence $a_1 \cdots a_l$ satisfying~$y_{w_{f}}<a_m<y_{w_i}$. 
Then, by the definition of $\rho$, the left side of rectangle~$a_m$ is contained in $W$.  
This implies that $a_m$ occurs between~$y_{w_i}$ and~$y_{w_{i+1}}$ in~$\sigma_{W_2}$, again a contradiction.
Thus every element of the sequence $a_1 \cdots a_l$ must be less than $y_{w_{i+1}}$ or greater than $y_{w_i}$.
Let~$a_m$ denote the first element of the sequence that satisfies $a_m<y_{w_{i+1}}$, if such an element exists.
In this case, (taking~$a_0=y_{w_i}$ if $m=1$) we reach a contradiction since $y_{w_1}a_{m-1}a_my_{w_{i+1}}y_{w_f}$ is an occurrence of the 3-51-2-4 pattern in~$y$.
Thus $a_m>y_{w_i}$ for all~$m$. 
However, this is also impossible since if $a_l \geq y_{w_i}$, then the subsequence~$y_{w_1}y_{w_i}a_ly_{w_{i+1}}y_{w_f}$ of~$y$ forms a 2-4-51-3 pattern in $y$.
Therefore, $y_{w_i}$ and~$y_{w_{i+1}}$ are adjacent in $y$.  
Let $x=y_1 \cdots y_{w_{i+1}}y_{w_i} \cdots y_n$.
Since $y_{w_1}y_{w_i}y_{w_{i+1}}y_f$ is an occurrence of the pattern \mbox{2-41-3} in $y$, by  Theorem \ref{thm:drecmove}, we have that $\rho(x)=\rho(y)$.
Now consider the wall shuffles of~$x$ and~$y$.    
Switching the order of $y_{w_i}$ and~$y_{w_{i+1}}$ in $y$ to obtain $x$ switches their order in the wall shuffle associated with~$W_2$ so $\sigma_{W_1}=y_{w_1}\cdots y_{w_{i+1}}y_{w_{i}}\cdots y_{w_f}$.
Every other wall shuffle of $R_2$ is unchanged in $\gamma(x)$ since $\rho(x)=\rho(y)$ and rectangles~$x_{w_i}$ and~$x_{w_{i+1}}$ are adjacent to no other shared wall.  
Thus~$\gamma(x)=R_1$.

\begin{figure}
\begin{tikzpicture}
\draw (12, .75) -- (13, .75);
\draw (13, 1.25) -- (14, 1.25);
\draw(12, -1)--(14, -1);
\draw(12, 3)--(14, 3);
\draw[dashed, line width = 1.5pt] (13, .75)--(13, 1.25);
\draw (13, -1) -- (13, .75);
\draw (13, 1.25) -- (13, 3);
\node(1) at (12.5,-.8) {$y_{w_1}$};
\node(2) at (13.5, 1) {$y_{w_i}$};
\node(3) at (12.5, 1) {$y_{w_{i+1}}$};
\node(4) at (13.5, 2.7) {$y_{w_f}$};
\node(5) at (13,-2) {$R_1$};
\draw (18, 1.25) -- (19, 1.25);
\draw (19, .75) -- (20, .75);
\draw(18, -1)--(20, -1);
\draw(18, 3)--(20, 3);
\draw[line width = 1.5pt, dashed] (19, .75)--(19, 1.25);
\draw (19, -1) -- (19, .75);
\draw (19, 1.25) -- (19, 3);
\node(1) at (18.5,-.8) {$y_{w_1}$};
\node(2) at (19.5, .5) {$y_{w_i}$};
\node(3) at (18.5, 1.5) {$y_{w_{i+1}}$};
\node(4) at (19.5, 2.7) {$y_{w_f}$};
\node(5) at (19,-2) {$R_2$};

\end{tikzpicture}
\caption{Diagrams used in the proof of Proposition \ref{prop:wallslide}.  
In each diagram, $y_{w_1}$ is the lowest rectangle on the left side of $W$ and $y_{w_f}$ is the uppermost rectangle on the right side of~$W$.   
No additional edges of $R_1$ or $R_2$ may be adjacent to the dashed segments.}
\label{fig:moves w/ labels}
\end{figure}
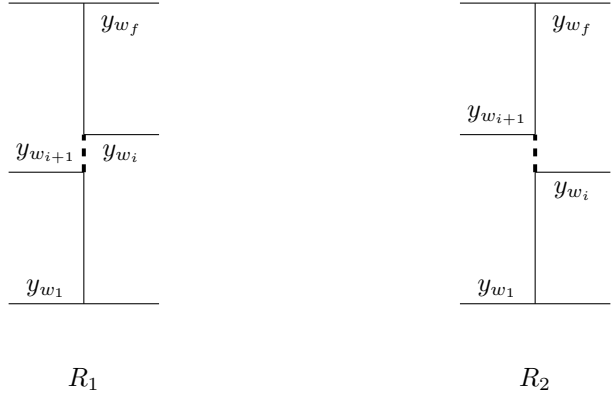

Now assume that $R_1$ and $R_2$ differ by a single horizontal wall slide such that the lower illustration of Figure \ref{fig:moves} corresponds to $R_1$ and the upper illustration corresponds to $R_2$.
By the definition of $\rf$, generic rectangulations $\rf(R_1)$ and $\rf(R_2)$ differ by a single vertical wall slide such that $\rf(R_1)$ contains the configuration shown in upper illustration of the fourth diagram of Figure \ref{fig:moves} and $\rf(R_2)$ contains the configuration in the lower illustration.
By the first part of this proof, $\rf(R_2)\lessdot \rf(R_1)$. 
Thus, by Lemma \ref{lemma:rf}, we have that $R_1 \lessdot R_2$.
\end{proof}

\section{The product and coproduct}
\label{sect:product}

In this section, we prove Theorems \ref{thm:product} and \ref{thm:coproduct}.

\begin{proof}[Proof of Theorem \ref{thm:product}] 
Let $x \in Cl_p^2$ and $y \in Cl_q^2$ such that $\gamma(x)=R_1$ and \mbox{$\gamma(y)=R_2$.}  
Corollary \ref{eq:prod} states that $x \bullet_{Cl_2} y =\sum [xy',y'x]$ where the summation denotes the sum of all elements of the interval $[xy',y'x]$ in the lattice of 2-clumped permutations of size $p+q$.  
Applying the bijection~$\gamma$ to this equation, we obtain $\gamma(x) \bullet_{gR} \gamma(y)=R_1\bullet_{gR} R_2=\sum[\gamma(xy'),\gamma(y'x)]$, where the summation denotes the sum of all elements of the interval $[\gamma(xy'),\gamma(y'x)]$ in $gRec_n$.
Applying $\gamma$ to $xy'$ and~$y'x$ results in the generic rectangulations~$R_1R_2'$ and $R_2'R_1$ respectively.  
\end{proof}
 
 To prove that the coproduct in $gRec$ is given by Theorem \ref{thm:coproduct} requires more work.
Applying~$\gamma$ to the equation in Theorem \ref{thm:Shcoprod} and noting that $\gamma(\pi_{\downarrow}^2(y))=\gamma(y)$ for any permutation~$y$, we first obtain the following corollary:
 
 \begin{corollary}\label{cor}
 Suppose $R\in gRec_n$ and $x\in Cl_n^2$ such that $\gamma(x)=R$.  Then
\begin{equation*}
\Delta_{gR}(R)=\sum_{\substack {T \text{ is good} \\ \text{with respect to } x}} I_T \otimes J_T
\end{equation*}
where $I_T$ is the sum of elements in the interval $[\gamma(\st(x_{\min}|_T)), \gamma(\st(x_{\max}|_T))]$ in $gRec_p$ and $J_T$ is the sum of elements in the interval $[\gamma(\st(x_{\min}|_{T^C})), \gamma(\st(x_{\max}|_{T^C}))]$ in $gRec_q$. 
\end{corollary}

Theorem \ref{thm:coproduct} will follow from Corollary \ref{cor}, and Lemmas~\ref{lemma:good}, \ref{lemma:coprod}, and \ref{lemma:coprod2}.
 In the proof of Lemma \ref{lemma:good}, we will demonstrate that for any $x\in Cl_n^2$ such that $\gamma(x)=R$ there is a natural correspondence between sets that are good with respect to $x$ and good paths in $R$.
Then, in the proofs of Lemmas \ref{lemma:coprod} and \ref{lemma:coprod2}, we will show that for each good set $T$ and corresponding good path $\Path$, we have $\gamma(\st(x_{\min}|_{T}))={R_l(\Path)}_|$,
$\gamma(\st(x_{\max}|_T))={R_l(\Path)}_-$,
$\gamma(\st(x_{\min}|_{T^C}))={R_u(\Path)}_|$, 
and $\gamma(\st(x_{\max}|_{T^C}))={R_u(\Path)}_-$.

We first make the following helpful observations about good sets.
Given $x \in Cl_n^2$ such that $\gamma(x)=R$, let $P$ be the partial order on $[n]$ such that the permutation $x'\in S_n$ is a linear extension of $P$ if and only if $\gamma(x')=R$.
We call $P$ the \emph{good set poset} of~$R$.
For each generic rectangulation, a good set poset exists because of a more general, well-known result.
 If $[a,b]$ is an interval in the right weak order on~$S_n$, then the elements of the interval are the linear extensions of the intersection of the total orders $a$ and $b$ (see for example \cite[Proposition 4.1]{Bax}).
Since each fiber of $\gamma$ forms an interval in the right weak order, for each generic rectangulation a good set poset exists.

An \emph{order ideal} $I$ of~$P$ is a subset of $P$ such that for each $x_i \in I$ if $x_j<_P x_i$, then $x_j\in I$.
The order ideals of the good set poset $P$ correspond exactly to the sets that are good with respect to $x$.
For each good set~$T$, let $P|_T$ denote the order ideal of~$P$ consisting of the elements of $T$.
The minimal linear extension of $P|_T$ is~$x_{\min}|_{T}$.
Similarly, the minimal linear extension of $P|_{T^C}$ is $x_{\min}|_{T^C}$, the maximal linear extension of $P|_{T}$ is $x_{\max}|_T$, and the maximal linear extension of~$P|_{T^C}$ is~$x_{\max}|_{T^C}$. 
To better understand the good sets associated with $x$, we describe the poset $P$. 


\begin{lemma}
 \label{lemma:P}
 Let $r_i$ and $r_j$ be rectangles of a generic rectangulation $R$ with $n$ rectangles, and~$P$ be the good set poset of $R$. 
 If $r_i$ comes before $r_j$ in some wall shuffle of $R$, then $r_i<_P r_j$. 
  Taking the transitive closure of these relations gives all of the relations in $P$.  
 \end{lemma}
 
 \begin{proof}
Given two permutations $x$ and $x'$ in $S_n$, we have that $R=\gamma(x)=\gamma(x')$ if and only if $\rho(x)=\rho(x')$ and the wall shuffles of $\gamma(x)$ are the same as the wall shuffles of $\gamma(x')$.  
Let $\rho(x)=D$ 
 and define the poset $Q$ on $[n]$ by declaring $r_i<_{Q}r_j$ if:
 \begin{itemize}
 \item In $D$, the right edge of rectangle $r_i$ and the left edge of $r_j$ intersect in their interiors, 
 \item  In $D$, the top edge of rectangle $r_i$ and the bottom edge of rectangle $r_j$ intersect in their interiors, or 
 \item In some wall shuffle of $R$, the entry $r_i$ precedes $r_j$
\end{itemize}
and then taking the transitive closure.
The first two bullets in the definition of $Q$ ensure that if $x$ and $x'$ are linear extensions of $Q$, then $\rho(x)=\rho(x')$.
The third item ensures that the wall permutations of $x$ and $x'$ agree.
By the definition of $\gamma$, the permutation $x'$ is a linear extension of $Q$ if and only if $\gamma(x')=\gamma(x)$. 
Thus to prove the lemma, it suffices to demonstrate that $r_i<_Pr_j$ if and only if $r_i<_Qr_j$.

Since the condition for $r_i<_Pr_j$ is identical to the final condition for $r_i<_Q r_j$, we have that $r_i<_P r_j$ implies $r_i<_Q r_j$.
For the other direction, first assume that in $D$ the right edge of rectangle $r_i$ intersects the interior of the left edge of rectangle~$r_j$ (so $r_i<_Q r_j$) along some vertical wall $W$.  
 As illustrated in the left diagram of Figure \ref{fig:D}, since $D$ is a \emph{diagonal} rectangulation, each of the edges extending to the left of the wall is above each of the edges extending to the right of the wall.
 This implies that either rectangle $r_i$ is the lowermost rectangle on the left side of the vertical wall separating the two rectangles (shown as the darker shaded region in the diagram) or rectangle $r_j$ is the uppermost rectangle on the right side of the wall (shown as the lightly shaded region).
 This implies that $r_i$ is the first entry of $\sigma_W$ or~$r_j$ is the final entry of~$\sigma_W$ so in either case, $r_i <_P r_j$.   
 Similarly, as illustrated in the right diagram of Figure \ref{fig:D}, if the top edge of rectangle $r_i$ intersects the interior of the bottom edge of rectangle $r_j$ along some horizontal wall $W$ in $D$, then we again see that $r_i$ precedes $r_j$ in $\sigma_W$ so $r_i<_P r_j$.  
Since the final condition for~$r_i<_Q r_j$ is identical to the condition for $r_i<_P r_j$ and any relationship that comes from the transitive closure in $Q$ also holds in $P$, we have that $r_i<_Q r_j $ implies $ r_i<_P r_j$. 
 \end{proof}

 \begin{figure}
\begin{tikzpicture}
\filldraw[light-gray] (2, 1.3) rectangle (4, 4);
\filldraw[gray] (0,0) rectangle (2, 2.6);

\draw (2,0)--(2,4);
\draw (0,0)--(4,0);
\draw (0,4)--(4,4);
\draw (0,3.5)--(2,3.5);
\draw(0,2.8)--(2,2.8);
\draw (0,2.6)--(2,2.6);
\draw (2,1.3)--(4,1.3);
\draw(2, 1)--(4,1);

\filldraw[light-gray] (9.5,2) rectangle (12,4);
\filldraw[gray] (8,0) rectangle (10.4,2);

\draw (8,2)--(12,2);
\draw (8,0)--(8,4);
\draw (12,0)-- (12, 4);
\draw (9,2)--(9,4);
\draw (9.5,2)--(9.5,4);
\draw (11.2,0)--(11.2,2);
\draw(10.7, 0)--(10.7, 2);
\draw (10.4,0)--(10.4,2);
\end{tikzpicture}
\caption{An illustration used in the proof of Lemma \ref{lemma:P}.}
\label{fig:D}
\end{figure}
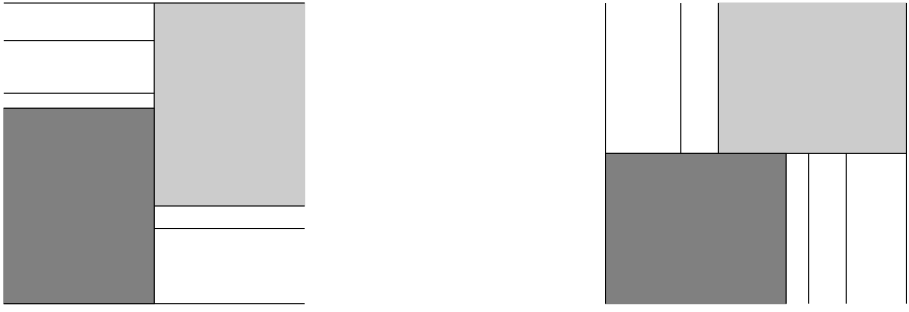 

 \begin{lemma}\label{lemma:good}
 Let $x\in Cl^2_n$ such that $\gamma(x)=R$.
 The set $T$ is good with respect to $x$ if and only if the union of the rectangles of $R$ labeled by elements of $T$ are exactly the rectangles below some good path $\Path$ in $R$.
 \end{lemma}
 
 \begin{proof}
 Let $x\in Cl^2_n$ such that $\gamma(x)=R$, the poset $P$ be the good set poset of $R$, and $T=\{t_1,...,t_p\}$ be a good set with respect to $x$ (i.e. an order ideal of $P$).
 If $T=\emptyset$, then the path $\Path$ passing above and left of the rectangles of $R$ labeled by elements of $T$ travels down the left side and then across the bottom of the square~$S$.
 This is a good path in $R$.
 
 Now suppose that $T\neq \emptyset$.
 Let $R_T$ denote the set of rectangles of $R$ labeled by elements of $T$. 
 To show that $R_T$ is the set of rectangles below some good path, we will show that:
 \begin{itemize}
 \item $R_T$ contains the bottom, left vertex of $S$, 
 \item $R_T$ is a connected set with no interior holes, and
 \item the path $\Path$ starting at the top, left corner of $S$, traveling along the left edge of~$S$ until it reaches the boundary of $R_T$, tracing the upper right boundary of $R_T$, and then traveling along the bottom of $S$ to the bottom right corner of $S$ is a good path.
 \end{itemize}
 
Since $x$ can be obtained from any permutation $x'$ such that $\gamma(x')=R$ by a sequence of (24513$\to$24153),  (42513$\to$42153), (35124$\to$31524), and (35142$\to$31542) moves, the first entry of $x$ is also the first entry of $x'$.
Thus by the definition of $\gamma$, some rectangle of $R_T$ contains the bottom, left vertex of $S$.

If $R_T$ is not connected, has an interior hole, or $\Path$ contains a left or up step, then the left side or bottom of some rectangle $t_i\in R_T$ intersects the boundary of some rectangle $u$ such that $u \in [n]-T$.
The two leftmost diagrams of Figure \ref{fig:good} illustrate these cases.
In each of the diagrams of Figure \ref{fig:good}, the shaded rectangles are contained in $R_T$.
If the left side of rectangle $t_i$ intersects the right side of rectangle~$u$ along a vertical wall $W$ (as illustrated in the leftmost diagram of Figure \ref{fig:good}), then the lower-right vertex of rectangle $u$ is below the upper-left vertex of rectangle $t_i$ on $W$.
Note that the lower-right vertex of rectangle $u$ is not necessarily contained in the left side of rectangle $t_i$ as shown in the diagram, but it is necessarily below the upper-left vertex of rectangle $t_i$.
Thus $u$ precedes $t_i$ in $\sigma_W$, contradicting the assumption that $T$ is an order ideal of $P$.
Similarly, if the bottom of rectangle~$t_i$ intersects the top of rectangle $u$ along a horizontal wall~$W$ (as illustrated in the second diagram of Figure \ref{fig:good}), then the upper-left vertex of rectangle $u$ is left of the lower-right vertex of rectangle $t_i$ on~$W$.
This also contradicts the assumption that~$T$ is an order ideal of $P$.

To complete the argument, we show that $\Path$ meets the two conditions for a good path.
Assume, for a contradiction, that the interior of a vertical segment of $\Path$ contains vertices $v$ and $v'$ of $R$ such that $v$ is the upper-left vertex of a rectangle~$u$ with $u\notin T$, vertex $v'$ is the lower-right vertex of a rectangle $t_i \in R_T$ and $v$ is below~$v'$.
This configuration is illustrated in the third diagram of Figure \ref{fig:good}.
The thick segment in the diagram is contained in $\Path$.
Since the upper-left vertex of rectangle~$u$ occurs below the bottom right vertex of rectangle $t_i$ along their shared wall, entry $u$ precedes $t_i$ in the associated wall shuffle, contradicting the assumption that $T$ is a good set.
Using the same reasoning, we conclude that the configuration illustrated in the rightmost diagram of Figure \ref{fig:good} also does not occur along $\Path$, that is, the interior of no horizontal segment of $\Path$ contains vertices $h$ and~$h'$ of $R$ such that $h$ is the lower-right vertex of a rectangle $u\notin T$, $h'$ is the upper-left vertex of a rectangle $t_i\in R_T$ and $h$ is left of $h'$.  
Thus the upper right border of $R_T$ determines a good path in $R$.  

\begin{figure}
\begin{tikzpicture}[scale=.8]
\filldraw[light-gray] (1,0) rectangle (2,2);
\draw(2,0)--(1,0)--(1,2)--(2,2);
\draw(0,1)--(1,1);
\node at (.5, 1.5) {$u$};
\node at (1.5, 1) {$t_i$};

\filldraw[light-gray] (4,1) rectangle (6,2);
\draw (4,2)--(4,1)--(6,1)--(6,2);
\draw (5,1)--(5,0);
\node at (5, 1.5) {$t_i$};
\node at (5.5, .5) {$u$};

\filldraw[light-gray] (9,0) rectangle (8,2);
\draw[line width = 2pt] (9, 2)--(9,0);
\draw (8, 1.5)--(9,1.5);
\draw (9, .5)--(10, .5);
\filldraw (9,.5) circle [radius=2pt];
\node at (8.7, .5) {$v$};
\filldraw (9,1.5) circle [radius=2pt];
\node at (9.3, 1.6) {$v'$};
\node at (9.5, .2) {$u$};
\node at (8.5, 1.8) {$t_i$};

\filldraw[light-gray] (12,0) rectangle (14,1);
\draw[line width = 2pt] (12, 1)--(14, 1);
\draw (12.5, 1)--(12.5, 2);
\draw(13.5, 1)--(13.5, 0);
\filldraw(12.5, 1) circle [radius=2pt];
\filldraw (13.5, 1) circle [radius=2pt];
\node at (12.5, .7) {$h$};
\node at (13.5, 1.3) {$h'$};
\node at (12.2, 1.5) {$u$};
\node at (13.8, .5) {$t_i$};
\end{tikzpicture}\caption{Diagrams for the proof of Lemma \ref{lemma:good}.}\label{fig:good}
\end{figure}
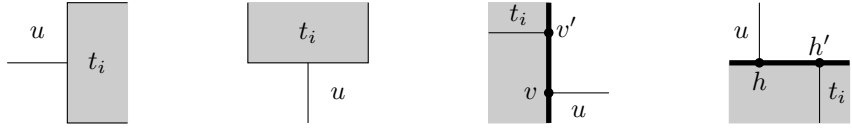  

Next we show that given any good path $\Path$ in $R$, the labels of the set of rectangles below and to the left of $\Path$, denoted by $T$, form a good set.
It is enough to demonstrate that $T$ is an order ideal of $P$, the good set poset of $R$.
For a contradiction, assume that $u\notin T$, $t_i\in T$, and $u$ precedes $t_i$ in $\sigma_W$, some wall shuffle of $R$. 
First let~$W$ be a vertical wall.
If rectangles~$t_i$ and $u$ are on the same side of $W$ or rectangle $u$ is on the left side of~$W$ while rectangle $t_i$ is on the right side of $W$, then~$\Path$ passes to the right of rectangle $t_i$ and then to the left of rectangle $u$, or below $u$ and then above $t_i$.
Thus $\Path$ contains a left step or an up step, a contradiction.
If rectangle $t_i$ is left of $W$ and rectangle $u$ is right of $W$, since the upper-left corner of rectangle $u$ is below the lower-right corner of rectangle $t_i$, we have that~$\Path$ contains a left step or violates the first condition of a good path.   
When $W$ is a horizontal wall, in each case we again reach a contradiction by showing that $\Path$ contains a left or up step, or violates the second condition of a good path.
%
 \end{proof}
 
For every good path $\Path$ of a generic rectangulation $R$, in the constructions of ${R_l(\Path)}_|$, ${R_l(\Path)}_-$, ${R_u(\Path)}_|$, and ${R_u(\Path)}_-$, the rectangles inherit a labeling (using the elements of $T$) from the labeling of $R$.
To simplify notation, in what follows, we do not standardize these labels.
In particular, when we refer to a permutation $x$ such that $\gamma(x)= {R_l(\Path)}_|$, this permutation~$x$ will be an ordering of the elements of~$T$ rather than an ordering of $\{1,...,|T|\}$.
To use $x$ to construct ${R_l(\Path)}_|$, we label the diagonal of $S$ with the elements of $T$ written in increasing order along the upper-left to bottom-right diagonal of $S$ and then construct $\gamma(x)$ as usual.
Additionally, we define the good set poset $P'$ of ${R_l(\Path)}_|$ to be the partial order on $T$ such that $x$ is a linear extension of $P'$ if and only if $\gamma(x)={R_l(\Path)}_|$.

Given a set $T$ that is good with respect to $x\in Cl_n^2$ such that $\gamma(x)=R$, we say that an ordering $t=t_1\cdots t_{|T|}$ of the elements of $T$ \emph{respects the ordering of the good set poset $P$ of $R$} if and only if there exists $x' =x_1' \cdots x_n' \in S_n$ such that $x_1' \cdots x_{|T|}'=t$ and $x'$ is a linear extension of $P$ (or equivalently $\gamma(x')=R$).
If some linear extension $t$ of a poset $P'$ respects the ordering of the good set poset $P$ of $R$ then we say that $P'$ is \emph{compatible} with~$P$.

\begin{lemma}\label{lemma:compat}
Let $R$ be a generic rectangulation, $\Path$ be a good path in $R$, poset $P$ be the good set poset of $R$, and $P'$ be the good set poset of $R_{l}(\Path)_|$.
Then $P'$ is compatible with $P$.
\end{lemma}

\begin{proof}
Let $T$ be the good set corresponding to $\Path$ (which exists by Lemma \ref{lemma:good}).
Assume that $P'$ is \emph{not} compatible with $P$ so there does not exist a linear extension of $P'$ that respects the ordering of $P$.
Since $T$ is a good set with respect to $R$, there exists an ordering of the elements of $T$ that respects the ordering of $P$.
If none of these orderings is a linear extension of $P'$ then there exist $r_j<_{P'} r_i$ such that $r_i<_P r_j$.
Below we show that this cannot occur by demonstrating that if $r_i, r_j \in T$ such that $r_i<_P r_j$ then $r_i<_{P'}r_j$.

To show that $r_i<_P r_j$ implies $r_i<_{P'} r_j$, it suffices to prove that if $r_i\lessdot_P r_j$ then $r_i<_{P'} r_j$.
Assume that $r_i, r_j\in T$ and $r_i \lessdot_P r_j$.
By Lemma \ref{lemma:P}, this implies that~$r_i$ immediately precedes $r_j$ in some wall shuffle $\sigma_W$ of $R$.
We consider cases and make use of the construction of ${R_l(\Path)}_|$.  

If rectangles $r_i$ and $r_j$ are on the same side of $W$, then the rectangles are adjacent, with rectangle $r_i$ left of or below rectangle $r_j$.
Assume that rectangles $r_i$ and $r_j$ are both above a horizontal wall $W$.
Since $r_i$ and $r_j$ are in $T$, path $\Path$ passes above both rectangles, so the bottom of both rectangles and the edge separating them are contained in $R_l(\Path)$.
Thus in $R_l(\Path)_|$, rectangles $r_i$ and $r_j$ are adjacent to a horizontal wall and $r_i$ precedes $r_j$ in that wall shuffle.
Therefore $r_i<_{P'}r_j$.
If rectangles $r_i$ and $r_j$ are both right of a vertical wall, the argument is similar.
Now assume that rectangles $r_i$ and $r_j$ are both below a horizontal wall~$W$.
Since~$r_i$ immediately precedes $r_j$ in $\sigma_W$, no vertical edge extends from the top of rectangle $r_i$.
Thus either path $\Path$ contains no part of the top of rectangle $r_i$ or $\Path$ contains the tops of rectangles $r_i$ and $r_j$.  
If $\Path$ contains no part of the top of rectangle $r_i$, then the top of rectangle $r_i$, part or all of the top of rectangle $r_j$, and the edge separating rectangles~$r_i$ and $r_j$ remain in $R_l(\Path)$.
To construct $R_l(\Path)_|$, the remaining portion of the top of rectangle $r_j$ is extended until it meets a vertical wall. 
Rectangles~$r_i$ and~$r_j$ are adjacent to the horizontal wall containing this extension.
Thus, $r_i$ precedes $r_j$ in this wall shuffle of $R_l(\Path)_|$ so $r_i<_{P'} r_j$.
If $\Path$ contains the tops of rectangles~$r_i$ and~$r_j$, then~$\Path$ contains points left of the upper-left corner of rectangle $r_i$ (as shown in the left diagram of Figure \ref{Rl-}) or the upper-left corner of rectangle $r_i$ is a vertex of~$\Path$ (as shown in the right diagram of Figure~\ref{Rl-}).
In these illustrations, two possible locations of $\Path$ are darkened.
The dotted segment of the second diagram may or may not be present in $R$.
In either of these cases, in $R_l(\Path)_|$ rectangle~$r_j$ extends to the top of~$S$ and rectangle $r_i$ is adjacent to the wall~$W'$ containing the left side of rectangle $r_j$.
Since rectangle $r_j$ is the uppermost rectangle on the right side of~$W'$, the final entry of $\sigma_{W'}$ is $r_j$.
Thus $r_i$ precedes $r_j$ in $\sigma_{W'}$ and so $r_i <_{P'} r_j$.   
If rectangles $r_i$ and $r_j$ are both adjacent to the left side of a vertical wall $W$, since $r_i$ and $r_j$ are adjacent in $\sigma_W$, no edge of $R$ extends from the right side of rectangle $r_i$.
Thus, regardless of the location of~$\Path$, in $R_l(\Path)_|$ the right sides of rectangles $r_i$ and~$r_j$ are contained in a single vertical wall and rectangle $r_i$ remains below rectangle $r_j$.
Therefore $r_i<_{P'} r_j$.

\begin{figure}  
\begin{tikzpicture}
\draw[line width = 2pt] (.5,0) -- (3,0);
\draw (1,-1)--(1,0);
\draw (2,-1)--(2,0);
\node at (1.5, -.5) {$r_i$};
\node at (2.5, -.5) {$r_j$};

\draw[line width = 2pt] (5,1)--(5,0)--(7,0);
\draw(5,-1)--(5,0);
\draw (6,-1)--(6,0);
\draw[thick, dotted] (4.5, .5)--(5, .5);
\node at (5.5, -.5) {$r_i$};
\node at (6.5, -.5) {$r_j$};
\end{tikzpicture}  
\caption{Configurations of rectangles $r_i$ and $r_j$ used in the proof of Lemma \ref{lemma:compat}.} \label{Rl-}
\end{figure}
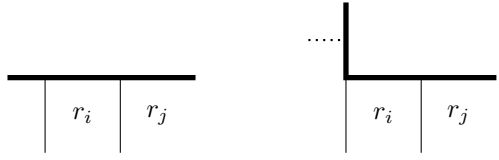

Now consider the case where rectangles $r_i$ and $r_j$ are on opposite sides of $W$.
If $W$ is horizontal, then rectangles $r_i$ and $r_j$ are in one of the two leftmost configurations shown in Figure \ref{Rl-2}.
Since~$r_j$ immediately follows $r_i$ in $\sigma_W$, no other edge can be adjacent to the dashed segment in the second diagram.
If rectangles $r_i$ and $r_j$ are in the first configuration of Figure~\ref{Rl-2}, then, regardless of the location of $\Path$, the left edge of rectangle $r_i$ and the horizontal edge between the rectangles remain in $R_l(\Path)$.
Thus by construction, rectangles $r_i$ and $r_j$ remain adjacent to $W$ in $R_l(\Path)_|$ with the upper-left vertex of rectangle $r_i$ to the left of the lower-right vertex of rectangle~$r_j$ so $r_i <_{P'} r_j$.
If rectangles $r_i$ and $r_j$ are in the second configuration of Figure \ref{Rl-2}, then we consider two cases.
First, if some part of the right side of rectangle $r_i$ is contained in $R_l(\Path)$ then some part of the top of rectangle $r_j$ is also contained in~$R_l(\Path)$.
Thus in $R_l(\Path)_|$, rectangles $r_i$ and $r_j$ remain adjacent to the extension of~$W$ with the lower-right vertex of rectangle $r_i$ to the left of the upper-left vertex of rectangle $r_j$.
Therefore $r_i <_{P'} r_j$.
If the right side of rectangle $r_i$ is contained in~$\Path$, then the dashed segment and the top of rectangle~$r_j$ are also contained in $\Path$.
In $R_l(\Path)_|$, the left edge of rectangle $r_j$ is extended to the top of~$S$ and the bottom of rectangle $r_i$ is extended to meet this vertical edge. 
Thus rectangles $r_i$ and $r_j$ are adjacent to opposite sides of a vertical edge of $R_l(\Path)_|$ with the bottom right vertex of rectangle~$r_i$ below the top left of rectangle $r_j$.
Therefore $r_i<_{P'}r_j$.

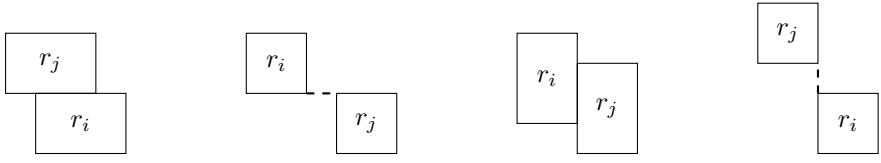
\begin{figure}
\begin{tikzpicture}[scale=.8]
\draw (0.5,1) rectangle (2,2);
\draw (1,0) rectangle (2.5,1);
\node at (1.75,.5) {$r_i$};
\node at (1.25,1.5) {$r_j$};

\draw (4.5,1) rectangle (5.5,2);
\draw (6,0) rectangle (7, 1);
\node at (5, 1.5) {$r_i$};
\node at (6.5, .5) {$r_j$};
\draw[thick, dashed] (5.5, 1)--(6,1);

\draw (10,0) rectangle (11,1.5);
\draw (9,.5) rectangle (10,2);
\node at (10.5, .75) {$r_j$};
\node at (9.5, 1.25) {$r_i$};

\draw(14,0) rectangle (15,1);
\draw (13, 1.5) rectangle (14, 2.5);
\node at (14.5, .5) {$r_i$};
\node at (13.5, 2) {$r_j$};
\draw[thick, dashed
] (14, 1)--(14,1.5);
\end{tikzpicture}
\caption{Additional illustrations used in the proof of Lemma \ref{lemma:compat}.}
\label{Rl-2}
\end{figure}

If rectangles $r_i$ and $r_j$ are on opposite sides of a vertical wall $W$, then they form one of the configurations shown in third or fourth diagram of Figure \ref{Rl-2}.
If they form the configuration shown in the third diagram, then the bottom edge of rectangle $r_i$ and the edge between the rectangles remain in $R_l(\Path)$.
Thus in $R_l(\Path)_|$, rectangles $r_i$ and $r_j$ are adjacent to the extension of $W$ with the bottom right vertex of rectangle~$r_i$ below the top left vertex of rectangle $r_j$ so $r_i<_{P'}r_j$.
If rectangles~$r_i$ and~$r_j$ form the configuration shown in the final diagram of Figure~\ref{Rl-2} and some part of the right side of rectangle $r_j$ is not contained in $\Path$, then some part of the top of rectangle~$r_i$ is also not contained in $\Path$.
In $R_l(\Path)_|$, the top edge of rectangle~$r_i$ remains below the bottom edge of rectangle $r_j$ so $r_i<_{P'}r_j$.
If instead the right side of rectangle $r_j$ is contained in~$\Path$, then the top of rectangle $r_i$ is also contained in~$\Path$.
In the construction of $R_l(\Path)_|$, the extension of the left side of rectangle $r_i$ is stopped by the $\epsilon$ extension of the bottom edge of rectangle~$r_j$.
Thus in $R_l(\Path)_|$, rectangles $r_i$ and $r_j$ are adjacent to the horizontal wall containing the extension of the bottom edge of rectangle~$r_j$ with the upper-left vertex of rectangle~$r_i$ left of the lower-right vertex of rectangle $r_j$ so $r_i<_{P'} r_j$.  
\end{proof}

\begin{lemma}\label{lemma:incompat}
Let $R\in gRec_n,$ let $P$ be the good set poset of $R$, and let $\Path$ be a good path in~$R$.
Let $\widetilde{R} \lessdot R_{l}(\Path)_|$ in $gRec_{|T|}$ and $\widetilde{P}$ be the good set poset of $\widetilde{R}$.
Then $\widetilde{P}$ is \emph{not} compatible with~$P$.
\end{lemma}

\begin{proof}
Let $T$ be the good set corresponding with good path $\Path$.
Again to simplify notation, we label each rectangle of $R_l(\Path)_|$ using the label (which is an element of $T$) inherited from $R$. 
By labeling the upper-left to lower-right diagonal of the square $S$ with the elements of $T$ (in numerical order), we also obtain a labeling of the rectangles of $\widetilde{R}$ by the elements of $T$.
Let $P'$ be the partial order on $T$ such that~$x$ is a linear extension of $P'$ if and only if $\gamma(x)=R_l(\Path)_|$ and~$\widetilde{P}$ be the partial order on $T$ such that $x$ is a linear extension of~$\widetilde{P}$ if and only if $\gamma(x)=\widetilde{R}$.
To show that $\widetilde{P}$ is not compatible with $P$, we demonstrate that no linear extension of~$\widetilde{P}$ respects the ordering of $P$ or equivalently that there exist $r_i, r_j$ in $T$ satisfying $r_j<_{\widetilde{P}}r_i$ such that $r_i<_P r_j$.

Since $\widetilde{R} \lessdot R_{l}(\Path)_|$ in ${gRec_{|T|}}$, by Theorem \ref{thm: covers} a wall slide or generic pivot is performed on ${R_l(\Path)}_|$ to obtain $\widetilde{R}$.
First assume that rectangles $r_i$ and $r_j$ of $R_l(\Path)_|$ form a configuration illustrated in one of the three leftmost upper diagrams of Figure~\ref{fig:moves} with $r_i<_{P'}r_j$ and that the edge $E$ which is pivoted to obtain $\widetilde{R}$ is completely contained in $R_l(\Path)$.
Since $E$ is completely contained in $R_l(\Path)$, rectangles~$r_i$ and~$r_j$ form this same configuration in $R$ and we have that $r_i<_P r_j$.
Pivoting $E$ to obtain~$\widetilde{R}$, we see that $r_j$ precedes $r_i$ in a wall shuffle of $\widetilde{R}$ so $r_j<_{\widetilde{P}} r_i$. 
Thus, in this case, $\widetilde{P}$ is not compatible with $P$.

We next consider the cases in which $R_l(\Path)_|$ and $\widetilde{R}$ differ by a wall slide.
If they differ by a horizontal wall slide, then in $R_{l}(\Path)_|$ rectangles $r_i$ and $r_j$ form the configuration shown in the left diagram of Figure \ref{fig:incompat}.
By the construction of~$R_l(\Path)_|$ (since no new vertical edges extending upward from a horizontal walls are created), the lower-right vertex and some portion of the right side of rectangle $r_j$ are contained in $R_l(\Path)$.
Additionally, since the upper-left vertex of rectangle~$r_i$ is left of the lower-right vertex of rectangle $r_j$ and $\Path$ is a good path, the left side of rectangle~$r_i$ is contained in $R_l(\Path)$.
Thus rectangles $r_i$ and $r_j$ form the same configuration in $R$, implying that $r_i <_P r_j$.
Performing a wall slide to obtain~$\widetilde{R}$ from~$R_l(\Path)_|$, we see that $r_j<_{\widetilde{P}} r_i$.
We conclude that in this case, $\widetilde{P}$ is not compatible with~$P$.
If rectangulations $R_l(\Path)_|$ and $\widetilde{R}$ differ by a vertical wall slide, then in $R_l(\Path)_|$ rectangles~$r_i$ and~$r_j$ form the configuration shown in the second diagram of Figure \ref{fig:incompat}.
If the lower-right vertex of rectangle $r_j$ were a constructed vertex, in the final step of the construction of $R_l(\Path)_|$, a wall slide would be performed to move the vertex below the upper-left vertex of rectangle $r_i$.
Thus, this configuration of rectangles would not appear in $R_l(\Path)_|$.
Therefore the lower-right vertex of rectangle $r_j$ is a vertex of $R_l(\Path)$.
Because $\Path$ is a good path, the upper-left vertex of rectangle $r_i$ is also a vertex of $R_l(\Path)$ so this configuration of rectangles $r_i$ and $r_j$ appears in $R$.
Thus $r_i<_P r_j$ and $r_j <_{\widetilde{P}} r_i$, implying that $\widetilde{P}$ is not compatible with~$P$.

\begin{figure}
\begin{tikzpicture}
\draw(0,0) --(2,0);
\draw (.5, -1)--(.5, 0);
\draw(1.5, 0)-- (1.5, 1);
\node at (1, .5) {$r_j$};
\node at (1, -.5) {$r_i$};

\draw (5, -1)--(5,1);
\draw (4, .25)--(5, .25);
\draw (5, -.25) -- (6, -.25);
\node at (4.5, .6) {$r_j$};
\node at (5.5, -.6) {$r_i$};
\end{tikzpicture}\caption{An illustration used in the proof of Lemma \ref{lemma:incompat}.}\label{fig:incompat}
\end{figure}
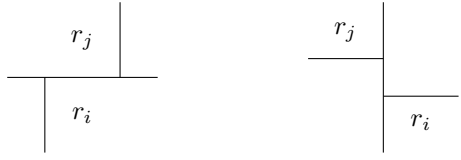

Finally, we consider the effect of performing a generic pivot on a horizontal edge~$E$ of ${R_l(\Path)}_|$ such that $E$ is not completely contained in $R_l(\Path)$.
There are two cases to consider: either $E$ is a new edge of ${R_l(\Path)}_|$ (in other words, no points of $E$ are contained in $R_l(\Path)$) or $E$ is the extension of some edge ${E'}$ of $R$. 

\begin{figure}
\begin{tikzpicture}

\draw (1,0)--(1,1);
\draw [ultra thick] (1,1)--(1,3);
\draw (0,2) -- (1,2);
\draw[line width= 2pt] (1,1)--(2,1);
\node at (.5, 1) {$r_k$};
\node at (.5, 2.5) {$r_j$};
\node at (1.5, .2) {$r_i$};
\node at (1.5, 2) {$r_l$};

\draw (5,0) --(5,2);
\draw(4,2) -- (5,2);
\draw(5,2) --(6,2);
\draw (6,0)--(6,3);
\node at (4.5, 1) {$r_k$};
\node at (5.5, 1) {$r_i$};
\node at (5, 2.5) {$r_j$};

\draw(9,0) --(9,3);
\draw(8,2) -- (9,2);
\draw (10,0) --(10,3);
\node at (8.5, 1) {$r_k$};
\node at (9.5, 1) {$r_i$};
\node at (8.5, 2.5) {$r_j$};
\end{tikzpicture}
\caption{Diagrams used in the proof of Lemma \ref{lemma:incompat}.
}
\label{fig:extension}
\end{figure}
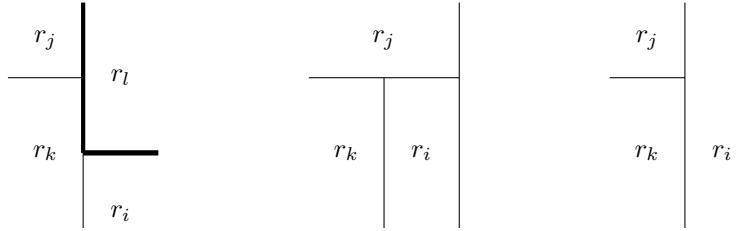

First consider the case where no points of $E$ are contained in $R_l(\Path)$.
By the construction of ${R_l(\Path)}_|$, edge $E$ results from a configuration in $R$ as shown in the leftmost diagram of Figure \ref{fig:extension}. 
 In the diagram, a subset of the rectangles of~$R$ are labeled $r_i, r_j, r_k, r_l$, and a portion of a good path $\Path$ is shown as a darkened segment.
The wall shuffle of the vertical wall shown contains the subsequence $r_k r_i r_j r_l$ so $r_i<_P r_j.$
To obtain ${R_l(\Path)}_|$, we remove $\Path$, extend the bottom of rectangle~$r_j$ by $\epsilon$ to the right, extend the left side of rectangle $r_i$ upwards until it hits the extension of the bottom of rectangle~$r_j$ and then extend the bottom of rectangle $r_j$ further until it reaches the extension of some vertical wall or the right side of $S$.
If necessary, we then perform wall slides along vertical walls, as described in the definition of~$R_l(\Path)_|$, but these wall slides do not affect the configuration of rectangles $r_i, r_j$, and $r_k$ in~$R_l(\Path)_|$ (shown in the center diagram of Figure \ref{fig:extension}).
Let $E$ be the edge of $R_l(\Path)_|$ that separates rectangles $r_i$ and~$r_j$.
Performing a generic pivot on $E$ to obtain $\widetilde{R}$ results in the configuration shown in the rightmost diagram of Figure \ref{fig:extension}.
In $\widetilde{R}$, the lower-right vertex of rectangle $r_j$ is below the upper-left vertex of rectangle $r_i$ along their shared wall, so $r_j$ precedes $r_i$ in this wall shuffle.
Thus $r_j<_{\widetilde{P}}r_i$, implying that $\widetilde{P}$ is not compatible with $P$.

Now consider the case where $E$ is the extension of some edge ${E'}$ of $R$.
This means that one endpoint~$v_0$ of $E'$ is contained in $R_l({\Path})$ and the other is on $\Path$.
In~$R$, let rectangle $r_i$ be below ${E'}$ and rectangle $r_j$ be above ${E'}.$
Thus the upper-left vertex of rectangle $r_i$ is left of the lower-right vertex of rectangle $r_j$ on the wall of~$R$ containing $E'$.
This implies that $r_i <_P r_j$. 
In~${R_l(\Path)}_|$, rectangles $r_i$ and $r_j$ are adjacent to $E$ with rectangle $r_i$ below rectangle $r_j$. 
By the construction of ${R_l(\Path)}_|$, the right endpoint of $E$ is the final vertex on the horizontal wall containing $E$.
So that $E$ can be pivoted to obtain $\widetilde{R}$ from ${R_l(\Path)}_|$, in ${R_l(\Path)}_|$ rectangles $r_i$ and~$r_j$ must form one of the configurations shown in Figure \ref{fig:flip}.
However, in both cases, pivoting~$E$ results in a rectangulation $\widetilde{R}$ in which $r_j$ precedes $r_i$ in a vertical wall shuffle.
Thus $r_j<_{\widetilde{P}} r_i$.
Therefore, regardless of the position of the generic pivot or wall slide used to obtain $\widetilde{R}$ from $R_l(\Path)_|$, the poset $\widetilde{P}$ is not compatible with $P$.
\end{proof}

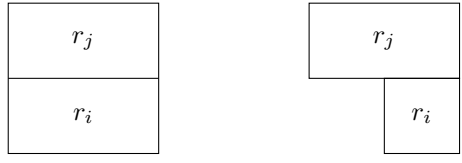
\begin{figure}
\begin{tikzpicture}
\draw (0,0) rectangle (2,2);
\draw(0,1) -- (2,1);
\node at (1,.5) {$r_i$};
\node at (1, 1.5) {$r_j$};

\draw (4,1) rectangle (6, 2);
\draw (5,0) rectangle (6,1);
\node at (5.5, .5) {$r_i$};
\node at (5, 1.5) {$r_j$};
\end{tikzpicture}
\caption{Configurations of $r_i$ and $r_j$ in ${R_l(\Path)}_|$ which allow for a generic pivot to be performed on the edge separating the rectangles.}
\label{fig:flip}
\end{figure} 

\begin{lemma}\label{lemma:coprod} 
Let $R \in gRec_n$ and $x\in Cl^2_n$ such that $\gamma(x)=R$.
For each set~$T$ that is good with respect to $x$ and corresponding good path $\Path$, we have that $\gamma(\st(x_{\min}|_{T}))={R_l(\Path)}_|$. \end{lemma}

\begin{proof}
%

Again let $P$ be the good set poset of~$R$ and $P'$ be the good set poset of~$R_l(\Path)_|$ (where $P'$ is a poset on $T$).
Let $$Y=\{y\in S_n\ |  \ \gamma(y)=R \text{ and } \{y_1,...,y_{|T|}\}=T\}.$$
The set of all permutations that map to $R$ under $\gamma$ and the set of all permutations whose first $|T|$ entries are the elements of $T$ each form a nonempty interval in the right weak order on $S_n$. 
Since $T$ is a good set, the intersection of these intervals is nonempty.
Thus, since the right weak order is a lattice, the elements of $Y$ form an interval in this lattice.
By definition, the minimal element of $Y$ is $x_{\min}$.

Let $$X=\{x'\in S_n \ | \ \gamma(x')=R \text{ and } x_1'\cdots x_{|T|}' \text{ is a linear extension of } P'\}.$$
By Lemma \ref{lemma:compat}, the set $X$ is non-empty.
Note that $X\subseteq Y$.
To prove the lemma, we wish to show that $x_{\min}\in X$.

To obtain a contradiction, assume that $x_{\min}\notin X$.
Thus, there exists some $y\in Y$ such that $y\notin X$ and $y$ is covered by an element of $X$.
Since $y\in Y$ and $y\notin X$, we have that $\gamma(\st(y|_T))\neq R_l(\Path)_|.$
Then $\gamma(\st(y|_T))$ is some $\widetilde{R}$ such that $\widetilde{R}\lessdot R_l(\Path)_|$ in~$gRec_{|T|}$.
By Lemma \ref{lemma:incompat}, the good set poset of $\widetilde{R}$ is not compatible with $P$.
This implies that $y\notin Y$, a contradiction.
\end{proof}

%

In what follows, we will consider distinct rectangulations and good sets simultaneously. 
To identify the rectangulation and good set used in each case, we will use the notation $x_{\max}(R, T), x_{\min}(R,T), R_l(\Path(R,T)),$ and  $R_u(\Path(R,T))$ where $R$ indicates the rectangulation of interest, $x\in Cl_n^2$ such that $\gamma(x)=R$, and $T$ is a set that is good with respect to $x$.
We will also make use of the maps $\rf, \rfu, \rp,$ and $\rv$.
Given a generic rectangulation $R$, recall that $\rf(R)$ is the reflection of $R$ about the upper-left to lower-right diagonal of the square $S$ and $\rfu(R)$ is the reflection of $R$ about the lower-left to upper-right diagonal of $S$.
Given a permutation $x=x_1\cdots x_n$, recall that $\rp(x)=x_n \cdots x_1$ and $\rv(x)=(n+1-x_1)\cdots(n+1-x_n).$

\begin{lemma}\label{lemma:rvrfu}
Let $R$ be a generic rectangulation, let $x\in Cl_n^2$ such that $\gamma(x)=R$, let $T=\{t_1,...,t_p\}$ a set that is good with respect to $x$, and let $T'=\{n+1-t_1,...,$ \mbox{$n+1-t_p\}$}.
Then $x_{\max}(R,T)|_{T}=\rv(x_{\min}(\rfu(R), T')|_{T'}).$
\end{lemma}

\begin{proof}
Let $P$ be  the good set poset of $R$ and $P'$ be the good set poset of $\rfu(R)$.
Since each wall shuffle $\sigma_W=x_{i_1}\cdots x_{i_s}$ of $R$ corresponds to a wall shuffle $\sigma_{W'}=(n+1-x_{i_1}) \cdots (n+1-x_{i_s})$ of $\rfu(R)$, we have that $x_i < x_j$ in $P$ if and only if $n+1-x_i<n+1-x_j$ in $P'$.
Because $T$ is a good set with respect to $R$, this implies that $T'$ is a good set with respect to $\rfu(R)$.
The order ideal $P|_T$ is isomorphic to the order ideal $P'|_{T'}$.
To find $x_{\max}(R,T)|_T$ an entry at a time using $P|_T$, at each step we consider the elements that have not yet been selected and are only greater than elements that have already been selected.
From this collection of elements, we choose the numerically largest value.
Analogously, to find $x_{\min}(\rfu(R), T')|_{T'}$ using~$P'|_{T'}$, we select the numerically smallest value from the candidate elements at each step.
Constructing $x_{\max}(R,T)|_T$ and $x_{\min}(\rfu(R), T')|_{T'}$ simultaneously, at each step the numerically largest candidate element of $P$ coincides with the numerically smallest candidate element of $P'$ under the poset isomorphism.
Thus applying~$\rv$ to $x_{\min}(\rfu(R), T')|_{T'}$ we obtain $x_{\max}(R,T)|_T$.
\end{proof}

\begin{lemma}\label{lemma:rprf}
Let $R\in gRec_n$, let $x\in Cl_n^2$ such that $\gamma(x)=R$ and let $T$ be a set that is good with respect to $x$. 
Then $x_{\min}(R,T)|_{T^C}=\rp(x_{\max}(\rf(R), T^C)|_{T^C})$.
\end{lemma}

\begin{proof}
Let $P$ be the good set poset of~$R$ and let $P'$ denote the good set poset of~$\rf(R)$.
Since applying $\rf$ to $R$ reverses each wall shuffle, we have that $P'$ is dual to $P$.
The set $T$ is an order ideal of $P$ so the set $T^C$ is a dual order ideal of $P$.
Thus $T^C$ is an order ideal of $P'$. 
Let $u_i, u_j \in T^C$ such that $u_i<u_j$ in numerical order.
Then $u_i$ precedes $u_j$ in $x_{\max}(\rf(R), T^C)|_{T^C}$ if and only if $u_i<_{P'} u_j$.
Equivalently, $u_i$ follows $u_j$ in $\rp(x_{\max}(\rf(R), T^C)|_{T^C})$ if and only if $u_i<_{P'}u_j$.
Entry $u_i$ follows $u_j$ in $x_{\min}(R, T)|_{T^C}$ if and only if $u_j<_P u_i$.
Since $P$ and $P'$ are dual posets, the result follows.
\end{proof}

\begin{lemma}\label{lemma:coprod2}
Let $R \in gRec_n$ and $x\in Cl^2_n$ such that $\gamma(x)=R$.
For each set $T$ that is good respect to $x$ and corresponding good path $\Path$, we have 
$\gamma(\st(x_{\max}|_T))={R_l(\Path)}_-$,
$\gamma(\st(x_{\min}|_{T^C}))={R_u(\Path)}_|$, 
and $\gamma(\st(x_{\max}|_{T^C}))={R_u(\Path)}_-$.
\end{lemma}

\begin{proof}
We use Lemma \ref{lemma:coprod} together with the maps $\rf, \rfu, \rp$ and $\rv$ to prove the equalities of this lemma.

To prove the first equality, we define $T'=\{n+1-t_i \ | \ t_i \in T\}$.
We now describe the manipulations that appear in (\ref{eq:2a})-(\ref{eq:2d}) below.
Using Lemma~\ref{lemma:rvrfu}, we obtain~(\ref{eq:2a}).
Since $\st \circ \rv= \rv \circ \st$ and $\gamma \circ \rv = \rfu \circ \gamma$, we obtain~(\ref{eq:2b}) from~(\ref{eq:2a}).
By Lemma \ref{lemma:coprod}, we have that (\ref{eq:2c}) follows.
Finally, since the constructions of the vertical and horizontal completions of $R_l$ are related by reflection about the bottom-left to upper-right diagonal, we obtain the desired result.
\begin{IEEEeqnarray}{rCl}
\IEEEyesnumber\label{eq:1} \IEEEyessubnumber*
\gamma(\st(x_{\max}(R, T)|_T)) & = & \gamma(\st(\rv(x_{\min}(\rfu(R),T')|_{T'})))\label{eq:2a}  \\
& = & \rfu(\gamma(\st(x_{\min}(\rfu(R), T')|_{T'}))) \label{eq:2b}\\
& = & \rfu(R_l(\Path(\rfu(R), T'))_|) \label{eq:2c}\\
& = & R_l(\Path(R, T))_- \label{eq:2d}
\end{IEEEeqnarray}

To prove the second and third equalities of the lemma, we first use Lemma~\ref{lemma:rprf} (see (\ref{3a}) and (\ref{4a})).
For (\ref{4a}), we apply the involution $\rf$ to make use of the equation in Lemma \ref{lemma:rprf}.
To obtain (b) from (a) in both manipulations we note that $\st \circ \rp = \rp \circ \st $ and $\gamma \circ \rp = \rf \circ \gamma$.
Then (\ref{3c}) follows from (\ref{3b}) by applying the first result of this lemma.
We obtain (\ref{3d}) since the construction of the horizontal completion of $R_l$ and the construction of the vertical completion of $R_u$ are related by reflection about the upper-left to bottom-right diagonal.  

\begin{IEEEeqnarray}{rCl}
\IEEEyesnumber\label{eq:2} \IEEEyessubnumber*
\gamma(\st(x_{\min}(R, T)|_{T^C})) & = & \gamma(\st(\rp(x_{\max}(\rf(R),T^C)|_{T^C})))\label{3a} \\
& = & \rf(\gamma(\st(x_{\max}(\rf(R), T^C)|_{T^C})))\label{3b} \\
& = & \rf(R_l(\Path(\rf(R), T^C))_-) \label{3c}\\
& = & R_u(\Path(R, T))_| \label{3d}
\end{IEEEeqnarray}

By Lemma \ref{lemma:coprod}, we have that (\ref{4c}) follows from (\ref{4b}).
Since the construction of the vertical completion of $R_l$ and the construction of the horizontal completion of $R_u$ are related by reflection about the upper-left to bottom-right diagonal, the final equality of this lemma follows.

\begin{IEEEeqnarray}{rCl}
\IEEEyesnumber\label{eq:3} \IEEEyessubnumber*
\gamma(\st(x_{max}(R, T)|_{T^C})) & = & \gamma(\st(\rp(x_{\min}(\rf(R),T^C)|_{T^C})))\label{4a}\\
& = & \rf(\gamma(\st(x_{\min}(\rf(R), T^C)|_{T^C}))) \label{4b}\\
& = & \rf(R_l(\Path(\rf(R), T^C))_|) \label{4c}\\
& = & R_u(\Path(R, T))_-\label{4d}
\end{IEEEeqnarray}
%
%
\end{proof}

Lemma \ref{lemma:coprod2} completes the proof of Theorem \ref{thm:coproduct}.



\end{document}